\documentclass[12pt]{amsart}

\usepackage{etex}
\usepackage{amsmath, amssymb}
\usepackage{array}
\usepackage[frame,cmtip,arrow,matrix,line,graph,curve]{xy}
\usepackage{graphpap, color, paralist, pstricks}
\usepackage[mathscr]{eucal}
\usepackage[pdftex]{graphicx}
\usepackage[pdftex,colorlinks,backref=page,citecolor=blue]{hyperref}
\usepackage{tikz}
\usepackage{tikz-cd}

\newtheorem{theorem}{Theorem}[section]
\newtheorem{theoremletter}{Theorem}

\newtheorem{proposition}[theorem]{Proposition}
\newtheorem{corollary}[theorem]{Corollary}
\newtheorem{lemma}[theorem]{Lemma}

\newtheorem{conjecture}[theorem]{Conjecture}

\theoremstyle{definition}
\newtheorem{definition}[theorem]{Definition}
\newtheorem{example}[theorem]{Example}

\newtheorem{remark}[theorem]{Remark}

\newcommand{\CC}{\mathbb{C} }
\newcommand{\PP}{\mathbb{P} }
\newcommand{\QQ}{\mathbb{Q} }
\newcommand{\RR}{\mathbb{R} }

\newcommand{\ZZ}{\mathbb{Z} }

\newcommand{\cA}{\mathcal{A} }
\newcommand{\cE}{\mathcal{E} }
\newcommand{\cF}{\mathcal{F} }

\newcommand{\cH}{\mathcal{H} }

\newcommand{\cM}{\mathcal{M} }

\newcommand{\cO}{\mathcal{O} }

\newcommand{\cV}{\mathcal{V} }

\newcommand{\rD}{\mathrm{D} }
\newcommand{\rH}{\mathrm{H} }
\newcommand{\rM}{\mathrm{M} }
\newcommand{\rN}{\mathrm{N} }

\newcommand{\ba}{\mathbf{a} }

\newcommand{\be}{\mathbf{e} }

\newcommand{\bm}{\mathbf{m} }
\newcommand{\bn}{\mathbf{n} }

\newcommand{\bx}{\mathbf{x} }

\newcommand{\spHom}{\mathrm{SParHom}}

\def\rHom{\mathrm{Hom} }
\def\CParHom{\mathcal{P}ar\mathcal{H}om }

\def\SL{\mathrm{SL}}

\def\Gr{\mathrm{Gr}}

\def\Ext{\mathrm{Ext} }

\def\git{/\!/ }
\def\Pic{\mathrm{Pic} }

\def\rk{\mathrm{rank}\, }
\def\im{\mathrm{im}\, }

\def\Quot{\mathrm{Quot}}
\def\proj{\mathrm{Proj}\;}
\def\Eff{\mathrm{Eff}}
\def\pdeg{\mathrm{pardeg}}

\def\Det{\mathrm{Det}}
\def\wt{\mathrm{wt}}

\begin{document}

\title[Derived category and ACM bundles of moduli space]{Derived category and ACM bundles of \\ moduli space of vector bundles on a curve}
\date{\today}

\dedicatory{Dedicated to the memory of M. S. Narasimhan}

\author{Kyoung-Seog Lee}
\address{Kyoung-Seog Lee, Institute of the Mathematical Sciences of the Americas, University of Miami, 1365 Memorial Drive, Ungar 515, Coral Gables, FL 33146}
\email{kyoungseog02@gmail.com}

\author{Han-Bom Moon}
\address{Han-Bom Moon, Department of Mathematics, Fordham University, New York, NY 10023}
\email{hmoon8@fordham.edu}

\maketitle

\begin{abstract}
We show that the derived category of a curve is embedded into the derived category of the moduli space of vector bundles on the curve of coprime rank and degree. We also generalize the semiorthogonal decomposition constructed by Narasimhan and Belmans-Mukhopadhyay. Finally, we produce a one-dimensional family of ACM bundles over the moduli space. 
\end{abstract}

\section{Introduction}\label{sec:intro}

The purpose of this paper is to give a complete affirmative answer to two problems on the moduli space of vector bundles on a curve. One is the embedding problem between derived categories, and the other is the construction of nontrivial arithmetically Cohen-Macaulay (ACM) bundles on the moduli space. 

Let $X$ be a smooth projective curve of genus $g \ge 2$. Fix two positive integers $r, d$ such that $(r, d) = 1$ and $0 < d < r$, and fix $L \in \Pic^{d}(X)$. The moduli space $\rM(r, L)$ of rank $r$, determinant $L$ stable vector bundles on $X$ is an $(r^{2}-1)(g-1)$-dimensional smooth Fano variety of index two. Let $\cE$ be the normalized Poincar\'e bundle on $X \times \rM(r, L)$. 

\subsection{Embedding of derived category}

We study the Fourier-Mukai transform $\Phi_{\cE} : \rD^{b}(X) \to \rD^{b}(\rM(r, L))$ with the kernel $\cE$. Narasimhan proved that $\Phi_{\cE}$ is an embedding when $r = 2$ in \cite{Nar17, Nar18}, by studying the Hecke correspondence. Fonarev and Kuznetsov proved the same result for a general $X$ using different techniques \cite{FK18}. Belmans and Mukhopadhyay extended Narasimhan's method and proved the embedding result for $r \ge 2$, $d = 1$ and $g \ge r + 3$ in \cite{BM19}. In this paper, we lift all the assumptions on the rank, degree and genus by employing birational geometry of moduli spaces of parabolic bundles, and the theory of derived categories of variation of GIT quotients, developed by Halpern-Leistner in \cite{HL15} and Ballard, Favero and Katzarkov in \cite{BFK19}.

\begin{theoremletter}\label{thm:mainthmembedding}
The functor $\Phi_{\cE} : \rD^{b}(X) \to \rD^{b}(\rM(r, L))$ is fully faithful. 
\end{theoremletter}

\subsection{Semiorthogonal decomposition}

To put our results in context, we explain a brief history of the question. After establishing his embedding theorem, Narasimhan made the following conjecture, which was announced in \cite{Lee18}. In the paper, as numerical evidence, Lee proved that the motive of $\rM(r,L)$ admits a compatible motivic decomposition. Further evidence was also provided by G\'omez and Lee in \cite{GL20}. Belmans-Galkin-Mukhopadhyay also independently proposed the same conjecture in \cite{BGM18}, with another numerical evidence in \cite{BGM20}.

\begin{conjecture}\label{conj:Narasimhan}
The category $\rD^{b}(\rM(2,L))$ has a semiorthogonal decomposition
\[
	\rD^b(\rM(2,L)) = \langle \{\rD^b(X_{k}), \rD^{b}(X_{k})\}_{0 \le k \le g-2}, \rD^b(X_{g-1}) \rangle,
\]
where $X_{k} = X^{k}/S_{k}$ is the $k$-th symmetric product of $X$.
\end{conjecture}

Toward the proof of Conjecture \ref{conj:Narasimhan}, Lee and Narasimhan showed that, by analyzing the Hecke correspondence, $\rD^{b}(X_{2})$ is embedded \cite{LN21} when $X$ is non-hyperelliptic and $g \geq 16$. After this, Tevelev-Torres and Xu-Yau showed that the above building blocks are embedded in $\rD^b(\rM(2, L))$, with entirely different approaches \cite{TT21, XY21}. At this point, the remaining part is whether these blocks span $\rD^{b}(\rM(2, L))$ or not. 

It is natural to guess the existence of a similar decomposition for $\rD^b(\rM(r,L))$ for general $r$ and $d$. Based on \cite{GL20}, where the motivic decomposition of $\rM(r, L)$ is studied, we expect the following statement. A more explicit version of the conjecture for $r = 3$ and its evidence can be found in \cite[Conjecture 1.9]{GL20}. 

\begin{conjecture}\label{conj:GomezLee}
The category $\rD^{b}(\rM(r,L))$ has a semi\-orthogonal decomposition where each indecomposable component is the derived category of products of $X_k$ and $\mathrm{Jac}(X)$. 
\end{conjecture}

As we can see from Brill-Noether theory, the geometry of curves is very complicated in general. However, in the study of $\rD^b(\rM(r, L))$, we expect a uniform decomposition, which does not depend on each curve.

Except for $g = r = 2$, it seems that $\rD^b(\rM(r,L))$ contains at least two copies of $\rD^b(\mathrm{pt})$ and two copies of $\rD^b(X)$. In \cite[Theorem B]{BM19}, the authors proved this is the case for $d = 1$ and (roughly) $g \ge 3r+4$. We extend this result for arbitrary coprime degree and give a constant genus bound. 

\begin{theoremletter}\label{thm:SOD}
If $g \ge 6$, there is a semiorthogonal decomposition
\[
	\rD^{b}(\rM(r, L)) = \langle \cA, {}^{\perp}\cA\rangle
\]
where $\cA = \langle \cO, \Phi_{\cE}(\rD^{b}(X)), \Theta, \Phi_{\cE}(\rD^{b}(X))\otimes \Theta\rangle$, where $\Theta$ is the ample generator of $\Pic(\rM(r, L))$. 
\end{theoremletter}

We expect that Theorem \ref{thm:SOD} is true without the genus restriction (except $g = r = 2$ case), but we have only partial result for $g \le 5$ (Theorem \ref{thm:threefactors} and Remark \ref{rmk:smallgenus}). 

\subsection{ACM bundles}

We also discover a family of ACM bundles on $\rM(r, L)$. For an $n$-dimensional projective variety $V$ with an ample line bundle $A$, a vector bundle $F$ is called \emph{ACM} if $\rH^{i}(V, F \otimes A^{j}) = 0$ for all $j \in \ZZ$ and $0 < i < n$. An ACM bundle $F$ is \emph{Ulrich} if $\rH^{0}(V, F \otimes A^{-1}) = 0$ and $\rH^{0}(V, F) = \mathrm{rank}\; F \cdot \deg V$. ACM bundles naturally appear in matrix factorization \cite{Eis80} and correspond to maximal Cohen-Macaulay modules in commutative algebra \cite{Yos90}. Ulrich bundles enable us to compute their associated Chow forms, and Eisenbud and Schreyer conjectured that every projective variety admits an Ulrich sheaf \cite{ES03}. However, since the above strong cohomology vanishing is difficult to expect and hard to verify, despite many works (see \cite{Bea18, CMRPL21, Fae13} and references therein), very few general results are known for higher dimensional varieties, even for the existence of ACM bundles except some trivial examples. 

\begin{theoremletter}\label{thm:ACM}
The restricted Poincar\'e bundle $\cE_x$ is ACM with respect to $\Theta$. Thus, there is a one-dimensional family of ACM bundles on $\rM(r, L)$, parametrized by $X$.
\end{theoremletter}

\subsection{Structure of the paper}

Sections \ref{sec:parabolicbundle} and \ref{sec:wallcrossing} review several basic results about the moduli space of parabolic bundles on a curve. In Section \ref{sec:nef}, we investigate the positivity of the restricted Poincar\'e bundle. In Section \ref{sec:mainexample}, we explicitly compute the wall-crossings in the case of two parabolic points. The result is essential in the following sections. Section \ref{sec:cohomologyderivedcategory} is devoted to calculating cohomology groups of certain line bundles via derived categories of the variation of GIT. Three main theorems are proved in the remaining sections.

\subsection*{Conventions}

We work over $\CC$. In this paper, $X$ denotes a smooth connected projective curve of genus $g \ge 2$. The moduli space of rank $r$, determinant $L$ (resp. degree $d$) semistable vector bundles is denoted by $\rM(r, L)$ (resp. $\rM(r, d)$). Unless stated explicitly, we assume $(r, d) = 1$, so $\rM(r, L)$ is a smooth projective variety. Let $\ell$ be the unique integer such that $\ell d \equiv 1 \;\mathrm{mod}\; r$ and $0 < \ell < r$. Let $\Theta$ be the ample generator on $\Pic(\rM(r, L))$. Let $\cE$ be the normalized Poincar\'e bundle on $X \times \rM(r, L)$ such that for each $x \in X$, its restriction $\cE_{x}$ to $x \times \rM(r, L) \cong \rM(r, L)$ has the determinant $\Theta^{\ell}$. For a vector space $W$, $\PP(W)$ is the projective space of one-dimensional quotients of $W$. Every algebraic stack is defined over the fppf topology. 

\subsection*{Acknowledgements}

The authors thank M. S. Narasimhan for drawing their attention to this problem, sharing his idea, and providing valuable suggestions about this and related projects. Especially the first author would like to express his deepest gratitude to him for his invaluable teaching and warm encouragement for many years. He also thanks Ludmil Katzarkov and Simons Foundation for partially supporting this work via Simons Investigator Award-HMS. Part of this work was done while the second author was visiting Stanford University. He gratefully appreciates the hospitality during his visit.


\section{Moduli spaces of parabolic bundles and their birational geometry}\label{sec:parabolicbundle}

This section explains the notion of parabolic vector bundles and their moduli space. This paper only considers the parabolic structure with at most one flag for each parabolic point. Fix a smooth connected projective curve $X$ and a finite ordered set $\bx := (x_{1}, x_{2}, \cdots, x_{k})$ of distinct closed points of $X$, so $(X, \bx) \in \cM_{g,n}$.  

\begin{definition}\label{def:parabolicbundle}
A rank $r$ \emph{parabolic bundle} over $(X, \bx)$ is a collection of data $(E, V_{\bullet})$ where
\begin{enumerate}
\item $E$ is a rank $r$ vector bundle over $X$;
\item $V_{\bullet} = (V_{1}, V_{2}, \cdots, V_{k})$ where $V_{i}$ is a subspace of $E|_{x_{i}}$. The dimension of $V_{i}$ is called the \emph{multiplicity} of $V_{i}$ and denoted by $m_{i}$. 
\end{enumerate}
The sequence $\bm = (m_{1}, m_{2}, \cdots, m_{k})$ is called the \emph{multiplicity} of $(E, V_{\bullet})$. 
\end{definition}

\begin{definition}\label{def:modulistackofparabolicbundle}
Let $\cM_{(X, \bx)}(r, L, \bm)$ (resp. $\cM_{(X, \bx)}(r, d, \bm)$) be the moduli stack of parabolic bundles $(E, V_{\bullet})$ over $(X, \bx)$ of rank $r$, determinant $L$ (resp. degree $d$), and multiplicity $\bm$. If there is no confusion, we use $\cM(r, L, \bm)$ (resp. $\cM(r, d, \bm)$). 
\end{definition}

This Artin stack is highly non-separated. To obtain a projective coarse moduli space that enables us to do projective birational geometry, we need to introduce a stability condition. 

For a parabolic bundle $(E, V_{\bullet})$, a \emph{parabolic subbundle} $(F, W_{\bullet})$ is a pair such that $F \subset E$ is a subbundle and $W_{i} = F|_{i}\cap V_{i}$. A \emph{parabolic quotient bundle} is defined as a parabolic bundle $(E/F, Y_{\bullet})$ such that $Y_{i} = \im (V_{i} \to E/F|_{i})$. A \emph{parabolic weight} $\ba = (a_{1}, a_{2}, \cdots, a_{k})$ is a sequence of rational numbers such that $0 < a_{i} < 1$. Intuitively, we may regard $\ba$ as extra weight for the parabolic flags. For a parabolic bundle $(E, V_{\bullet})$, its \emph{parabolic degree} is $\pdeg (E, V_{\bullet}) := \deg E + \sum_{1\le i \le k}m_{i}a_{i}$. The same parabolic weight can induce the parabolic degree for parabolic subbundles and parabolic quotient bundles of $(E, V_{\bullet})$. The \emph{parabolic slope} is $\mu(E, V_{\bullet}) := \pdeg (E, V_{\bullet})/\rk E$. 

\begin{definition}\label{def:stability}
Fix a parabolic weight $\ba$. A parabolic bundle $(E, V_{\bullet})$ is \emph{$\ba$-(semi)stable} if for every parabolic subbundle $(F, W_{\bullet})$, $\mu(F, W_{\bullet}) \;(\le ) < \mu(E, V_{\bullet})$. A parabolic weight $\ba$ is \emph{general} if the $\ba$-semistability coincides with the $\ba$-stability. 
\end{definition}

\begin{definition}\label{def:modulispaceparabolicbundle}
Let $(X, \bx) \in \cM_{g,n}$ with $g \ge 2$. Let $\cM(r, L, \bm, \ba)$ (resp. $\cM(r, d, \bm, \ba)$) be the moduli stack of rank $r$, determinant $L$ (resp. degree $d$), $\ba$-semistable parabolic bundles over $(X, \bx)$. Let $\rM(r, L, \bm, \ba)$ (resp. $\rM(r, d, \bm, \ba)$) be its good moduli space, which is a normal projective variety of dimension $(r^{2}-1)(g-1) + \sum m_{i}(r-m_{i})$ (resp. $r^{2}(g-1)+1 + \sum m_{i}(r-m_{i})$). When $\ba$ is general, both $\rM(r, L, \bm, \ba)$ and $\rM(r, d, \bm, \ba)$ are nonsingular.
\end{definition}

\begin{remark}
When $g \le 1$, the moduli space behaves differently. For instance, if $g = 0$, depending on $\ba$, $\cM(r, L, \bm, \ba)$ may be empty. Consult \cite{MY21}. 
\end{remark}

\begin{example}\label{ex:smallweight}
The inequality $\mu(F, W_{\bullet}) \le \mu(E, V_{\bullet})$ defining the $\ba$-semistability can be understood as a perturbation of the inequality $\mu(F) \le \mu(E)$ for the semistability of the underlying bundle. If $(r, d = \det L) = 1$, the inequality is strict for all $F \subset E$, and each coefficient of $\ba$ is sufficiently small and general, then $\ba$ does not affect the stability. Therefore, a parabolic bundle $(E, V_{\bullet})$ is $\ba$-stable if and only if the underlying bundle $E$ is stable. Thus, the forgetful morphism $\cM(r, L, \bm, \ba) \to \cM(r, L)$ induces a map between coarse moduli spaces 
\[
	\pi : \rM(r, L, \bm, \ba) \to \rM(r, L)
\]
and $\pi$ is a $\times \Gr(m_{i}, r)$-fibration. Indeed, for a fixed Poincar\'e bundle $\cE$ over $X \times\rM(r, L)$, 
\[
	\rM(r, L, \bm, \ba) \cong \times_{\rM(r, L)}\Gr(m_{i}, \cE_{x_{i}}).
\]
\end{example}

\begin{example}\label{ex:forgetfulmap}
More generally, if $\ba = (a_{i})$ is general and one $a_{i}$ is sufficiently small, then forgetting one flag does not affect on the stability calculation. Thus, there is a forgetful morphism 
\[
	\pi : \rM_{(X, \bx)}(r, L, \bm, \ba) \to \rM_{(X, \bx')}(r, L, \bm', \ba')
\]
where $\bx' = \bx \setminus \{x_{i}\}$, $\bm' = \bm \setminus \{m_{i}\}$, and $\ba' = \ba \setminus \{a_{i}\}$. This is a $\Gr(m_{i}, r)$-fibration. 
\end{example}

\begin{example}\label{ex:topandzerodimflags}
Fix $(X, \bx) \in \cM_{g,n}$. Let $\bx' := \bx \setminus \{x_{k}\}$. Let $\bm' = (m_{i})_{1 \le i \le k-1}$ and $\ba' = (a_{i})_{1 \le i \le k-1}$. Suppose that $m_{k} = 0$ or $r$. Then 
\[
	\rM_{(X, \bx)}(r, L, \bm, \ba) \cong \rM_{(X, \bx')}(r, L, \bm', \ba'). 
\]
\end{example}

When one of $\ba$ is sufficiently close to one, there is another contraction. 

\begin{proposition}\label{prop:generalizedHeckemodification}
We use the notation in Example \ref{ex:topandzerodimflags}. For a general parabolic weight $\ba = (a_{i})$, assume that $a_{k}$ is sufficiently close to one. Then there exists a morphism 
\[
	\pi_{1} : \rM_{(X, \bx)}(r, L, \bm, \ba) \to \rM_{(X, \bx')}(r, L(-(r-m_{k})x_{k}), \bm', \ba').
\]
\end{proposition}

\begin{proof}
It is sufficient to construct a morphism 
\[
	\cM_{(X, \bx)}(r, L, \bm, \ba) \to \cM_{(X, \bx')}(r, L(-(r-m_{k})x_{k}), \bm', \ba')
\]
between algebraic stacks. 

Let $\widetilde{\bm} = (\widetilde{m}_{i})$ be a multiplicity such that $\widetilde{m}_{i} = m_{i}$ for $1 \le i \le k-1$ and $\widetilde{m}_{k} = r$. By Example \ref{ex:topandzerodimflags}, there is a functorial isomorphism $\cM_{(X, \bx)}(r, L(-(r-m_{k})x_{k}), \widetilde{\bm}, \ba) \cong \cM_{(X, \bx')}(r, L(-(r-m_{k})x_{k}), \bm', \ba')$. Thus, it is sufficient to show that there is a morphism 
\[
	\cM(r, L, \bm, \ba) \to \cM(r, L(-(r-m_{k})x_{k}), \widetilde{\bm}, \ba).
\]

For a stable bundle $(E, V_{\bullet}) \in \cM(r, L, \bm, \ba)$, let $E'$ be the kernel of the restriction map $E \to E|_{x_{k}} \to E|_{x_{k}}/V_{x_{k}}$. Then for each $i \ne k$, $E'|_{x_{i}}$ can be identified with $E|_{x_{i}}$. Set $V_{i}' = V_{i}$ under this identification. On the other hand, the restriction $f : E'|_{x_{k}} \to E|_{x_{k}}$ is a linear map with image $V_{k}$. We set $V_{k}' := f^{-1}(V_{k}) = E'|_{x_{k}}$. Then we obtain a parabolic bundle $(E', V_{\bullet}') \in \cM(r, L(-(r-m_{k})x_{k}), \widetilde{\bm}, \ba)$. Thus, we have a morphism 
\begin{equation}\label{eqn:generalizedHecke}
\begin{split}
\cM(r, L, \bm, \ba) &\to \cM(r,L(-(r-m_{k})x_{k}), \widetilde{\bm})\\
(E, V_{\bullet}) & \mapsto (E', V_{\bullet}').
\end{split}
\end{equation}

We claim that $(E', V_{\bullet}')$ is $\ba$-semistable. Then the morphism in Equation \eqref{eqn:generalizedHecke} factors through $\cM(r,L(-(r-m_{k})x_{k}), \widetilde{\bm}, \ba)$. 

Suppose not. Then there is a parabolic subbundle $(F', W_{\bullet}')$ of $(E', V_{\bullet}')$ such that $\mu(F', W_{\bullet}') > \mu(E', V_{\bullet}')$. Let $\rk F' = s$, $\deg F' = e$, and $n_{i} = \dim W_{i}'$. Note that $n_{k} = s$. 

Set $d = \det L$. Then
\begin{equation}\label{eqn:slopecomparison}
\begin{split}
	\mu(E, V_{\bullet}) - \mu(E', V_{\bullet}') &= \frac{d + \sum m_{i}a_{i}}{r} - \frac{d - (r-m_{k}) + \sum_{i \ne k}m_{i}a_{i} + ra_{k}}{r}\\ 
	&= \frac{(r-m_{k})(1-a_{k})}{r}.
\end{split}
\end{equation}

In general, $F'$ is not a subbundle of $E$. But there is a subbundle $F$ of $E$ such that $F/F'$ is a sheaf supported on $x_{k}$ and $\dim (F/F')|_{x_{k}} = s - c$, where $c := \dim F|_{x_{k}} \cap V_{x_{k}}$. For the induced parabolic subbundle $(F, W_{\bullet})$ of $(E, V_{\bullet})$, 
\begin{equation}\label{eqn:slopecomparisonsubbundle}
\begin{split}
	\mu(F, W_{\bullet}) - \mu(F', W_{\bullet}') &= \frac{e + (s-c) + \sum_{i\ne k}a_{i}n_{i} + a_{k}c}{s} - \frac{e + \sum_{i \ne k}a_{i}n_{i} + a_{k}s}{s}\\
	&= \frac{(s-c)(1-a_{k})}{s}.
\end{split}
\end{equation}

By combining \eqref{eqn:slopecomparison} and \eqref{eqn:slopecomparisonsubbundle}, we have 
\[
	\mu(E, V_{\bullet}) - \mu(F, W_{\bullet}) = \mu(E', V_{\bullet}') - \mu(F', W_{\bullet}') + (1-a_{k})\left(\frac{r-m_{k}}{r} - \frac{s - c}{s}\right).
\]
Note that $\mu(E', V_{\bullet}') - \mu(F', W_{\bullet}')$ is independent from $a_{k}$, as the coefficient of $a_{k}$ in each term is one. Thus, if $a_{k}$ is sufficiently close to one, then the last term is negligible. By the assumption, $\mu(E', V_{\bullet}') - \mu(F', W_{\bullet}') < 0$ and hence the left hand side is also negative. It violates the stability of $(E, V_{\bullet})$ and obtain a contradiction. 
\end{proof}

\begin{remark}\label{rmk:Hecke}
The morphism in Proposition \ref{prop:generalizedHeckemodification} can be understood as a \emph{generalized Hecke correspondence}. When $d = k = 1$ and $m = r-1$, up to taking the dual bundle, we obtain the classical Hecke correspondence in the sense of \cite[Section 4]{NR75}. A difference in the $d > 1$ case is that $\rM(r, L, m, a)$ does not admit morphisms to both $\rM(r, L)$ and $\rM(r, L(-x))$, so we need a birational modification on $\rM(r, L, m, a)$. It can be explained in terms of parabolic wall-crossing as Section \ref{sec:wallcrossing} below. 
\end{remark}


\section{Wall-crossing}\label{sec:wallcrossing}

This section reviews how $\rM(r, L, \bm, \ba)$ changes as $\ba$ varies. 

\subsection{General theory}\label{ssec:generalthy}

Let $k$ be the number of parabolic points. Recall that a parabolic weight is, a length $k$ sequence of rational number $\ba = (a_{i})$ with $0 < a_{i} < 1$. The closure of the set of parabolic weights is $[0, 1]^{k} \subset \RR^{k}$. 

There is a wall-chamber decomposition of $[0, 1]^{k}$. A parabolic bundle $(E, V_{\bullet}) \in \rM(r, L, \bm, \ba)$ is strictly semi-stable if and only if there is a maximal destabilizing subbundle $(F, W_{\bullet})$ such that $\mu(F, W_{\bullet}) = \mu(E, V_{\bullet})$. More explicitly, this is true only if
\begin{equation}\label{eqn:wallequation}
	\frac{e+\sum n_{i}a_{i}}{s} = \frac{d+\sum m_{i}a_{i}}{r}
\end{equation}
for some $0 < s < r$, $e \in \ZZ$, and $\bn = (n_{i})$. Here $s$ is the rank, $e$ is the degree, and $\bn$ is the multiplicity of $(F, W_{\bullet})$. So we require that $n_{i} \le \mathrm{min}\;\{s, m_{i}\}$. Let $\Delta(s, e, \bn)$ be the set of weights that satisfy \eqref{eqn:wallequation}. Note that this is an intersection of a hyperplane and $[0, 1]^{k}$. We call $\Delta(s, e, \bn)$ a \emph{wall} if it is nonempty. We also obtain 
\begin{equation}\label{eqn:wallduality}
	\Delta(s, e, \bn) = \Delta(r-s, d-e, \bm - \bn).
\end{equation}
Note that $\Delta(s, e, \bn) = \Delta(ks, ke, k\bn)$ if $ks < r$ for some $k > 1$. We call such a wall a \emph{multiple wall}, and otherwise, it is a \emph{simple wall}. A wall $\Delta(s, e, \bn)$ is simple if and only if $\{s, e, n_{i}\}$ are coprime and $\{r-s, d-e, m_{i} - n_{i}\}$ are coprime. 

The stability changes only if a parabolic weight $\ba$ lies on one of the walls. So for each open chamber $C \subset [0, 1]^{k}$, for any $\ba, \ba' \in C$, $\rM(r, L, \bm, \ba) \cong \rM(r, L, \bm, \ba')$. The stability coincides with the semistability if $\ba \in (0, 1)^{k} \setminus \bigcup \Delta(s, e, \bn)$. 

Let 
\begin{equation}\label{eqn:wallcrossing}
	\xymatrix{\rM(r, L, \bm, \ba^{-}) \ar@{<-->}[rr] \ar[rd]_{\pi_{-}}&& \rM(r, L, \bm, \ba^{+}) \ar[ld]^{\pi_{+}}\\
	&\rM(r, L, \bm, \ba)}
\end{equation}
be a wall-crossing. Suppose that $\ba$ is a general point of $\Delta(s, e, \bn)$, and $\ba^{-}$ and $\ba^{+}$ are two very close weights on the opposite chambers. The contraction maps $\pi_{\pm}$ are birational surjections. Let $Y^{\pm}$ be the exceptional locus on $\rM(r, L, \bm, \ba^{\pm})$ and let $Y := \pi_{\pm}(Y^{\pm})$. The subvarieties $Y^{\pm}$ are called the \emph{wall-crossing centers}. For our purpose, we need a lower bound of the codimension of $Y^{\pm}$. Observe that the parabolic bundles in $Y^{-}$ are stable with respect to $\ba^{-}$, but unstable with respect to $\ba^{+}$. Thus, $Y^{-}$ parametrizes unstable parabolic bundles with respect to some weight. The codimension of the unstable locus is estimated in \cite{Sun00}. For an outline of the proof, see also \cite[Section 3.2]{MY20}.

\begin{theorem}[\protect{\cite[Proposition 5.1]{Sun00}}]\label{thm:codimunstable}
In $\cM(r, L, \bm)$, the codimension of the unstable locus with respect to a weight $\ba$ is at least $(r-1)(g-1) + 1$. 
\end{theorem}

\begin{corollary}\label{cor:codimcenter}
The codimension of the wall-crossing center is at least $(r-1)(g-1) + 1$. In particular, if $g \ge 2$, every wall-crossing is a flip. 
\end{corollary}

We say a wall-crossing is \emph{simple} if:
\begin{enumerate}
\item The wall $\Delta(s, e, \bn)$ is a simple wall and;
\item $\ba \in \Delta(s, e, \bn)$ is on a unique wall. 
\end{enumerate}
A simple wall-crossing has an explicit description. The wall-crossing centers $Y^{\pm}$ are irreducible and their image $Y \cong \rM(s, e, \bn, \ba) \times_{\mathrm{Pic}(X)}\rM(r-s, d-e, \bm - \bn, \ba)$ is a smooth variety. For $(E, V_{\bullet}) \in Y^{-}$, there is a unique maximal $\ba$-destabilizing subbundle $(E^{-}, V_{\bullet}^{-}) \in \rM(s, e, \bn, \ba)$, which fits into an exact sequence 
\[
	0 \to (E^{-}, V_{\bullet}^{-}) \to (E, V_{\bullet}) \to (E^{+}, V_{\bullet}^{+}) \to 0
\]
of parabolic bundles. The map $\pi_{-}$ is restricted to the map $Y^{-} \to Y$, which sends $(E, V_{\bullet})$ to the $S$-equivalence class of $(E, V_{\bullet})$, which is the class of $(E^{-}, V_{\bullet}^{-})\oplus(E^{+}, V_{\bullet}^{+})$. We denote this class by $((E^{-}, V_{\bullet}^{-}), (E^{+}, V_{\bullet}^{+}))$. Conversely, if $ x := ((E^{-}, V_{\bullet}^{-}), (E^{+}, V_{\bullet}^{+}))$ is a general point in $Y$ so that both $(E^{-}, V_{\bullet}^{-})$ and $(E^{+}, V_{\bullet}^{+})$ are stable, then the fiber $\pi_{-}^{-1}(x)$ is a projective space $\PP \Ext^{1}((E^{+}, V_{\bullet}^{+}), (E^{-}, V_{\bullet}^{-}))$ (see \cite[Section 1]{Yok95} for the derived functors on the category of parabolic sheaves). A functorial description is possible. Let $(\cE^{-}, \cV_{\bullet}^{-})$ (resp. $(\cE^{+}, \cV_{\bullet}^{+})$) be the Poincar\'e family over $\rM(s, e, \bn, \ba)$ (resp. $\rM(r-s, d-e, \bm- \bn, \ba)$). The standard GIT construction and the descent method imply the existence of Poincar\'e bundle (\cite[Chapter 5]{New78}, \cite[Section 4.6]{HL10}). Then $Y^{-} \cong \PP R^{1}\pi_{- *}\CParHom((\cE^{+}, \cV_{\bullet}^{+}), (\cE^{-}, \cV_{\bullet}^{-}))$ and $Y^{+} \cong \PP R^{1}\pi_{+ *}\CParHom((\cE^{-}, \cV_{\bullet}^{-}), (\cE^{+}, \cV_{\bullet}^{+}))$. In particular, they are projective bundles over $Y$. Finally, it is well-known that the blow-up of $\rM(r, L, \bm, \ba^{-})$ along $Y^{-}$ is isomorphic to the blow-up of $\rM(r, L, \bm, \ba^{+})$ along $Y^{+}$.

\subsection{GIT construction of moduli space}\label{ssec:GITconstruction}

The moduli spaces $\rM(r, L, \bm, \ba)$ can be constructed by GIT and each wall-crossing is indeed obtained by the variation of GIT. We review a GIT construction of $\rM(r, L, \ba)$ after Bhosle (\cite{Bho89}). 

We fix a degree one line bundle $\cO(1)$ on $X$. Fix an integer $m \gg 0$ such that $\rH^{1}(E(m)) = 0$ and $E(m)$ is globally generated for every $(E, V_{\bullet}) \in \rM(r, L, \ba)$. (Indeed, we may find such an $m$ that works for all $\ba$.) Let $\chi := \rH^{0}(E(m)) = d + r(m + 1-g)$ and let $Q := \Quot(\cO_{X}^{\chi})$ be the quot scheme parametrizing quotients of $\cO_{X}^{\chi}$ whose Hilbert polynomial is that of $E(m)$. Let $R \subset Q$ be a locally closed subscheme parametrizing the quotients $\cO_{X}^{\chi} \stackrel{\varphi}{\to} F \to 0$ such that $\rH^{1}(F) = 0$, $\rH^{0}(\cO_{X}^{\chi}) \stackrel{\varphi}{\cong} \rH^{0}(F)$, $F$ is locally free, and $\det F \cong L(rm)$. Let $\cO_{R \times X}^{\chi} \to \cF \to 0$ be the universal quotient over $R \times X$. For $x_1, x_2, \cdots, x_k \in X$, let
\[
	\widetilde{R} := \times_{R} \mathrm{Gr}(m_i, \cF|_{x_i})
\]
be the fiber product of Grassmannian bundles over $R$. There is a natural $\SL_{\chi}$-action on $\widetilde{R}$. Note that $\widetilde{R}$ parametrizes pairs $([\cO_{X}^{\chi} \stackrel{\varphi}{\to} F \to 0], \{V_{i} \subset F|_{x_{i}}\})$. 

We can make an explicit $\SL_{\chi}$-equivariant embedding of $\widetilde{R}$ into a product of elementary varieties as the following. Let $Z := \PP\rHom(\wedge^{r}\CC^{\chi}, \rH^{0}(L(rm)))^{*}$. Then for any $[\cO_{X}^{\chi} \stackrel{\varphi}{\to} F \to 0] \in R$, 
\[
	\wedge^{r}\CC^{\chi} \stackrel{\wedge^{r}\varphi}{\cong} \wedge^{r}\rH^{0}(F) \to \rH^{0}(\wedge^{r}F) \cong \rH^{0}(L(rm))
\]
gives a point in $Z$. Furthermore, for each $x_{i}$, by taking the inverse image $\psi_{i}^{-1}(V_{i})$ for 
\[
	\psi_{i} : \CC^{\chi} \cong \rH^{0}(F) \to F|_{x_{i}},
\]
we obtain an element in $\Gr(\chi - r + m_i, \chi)$ for $x_k$. Therefore, we have an $\SL_{\chi}$-equivariant morphism 
\begin{equation}\label{eqn:embeddingofZ}
	\widetilde{R} \to \times_Z \Gr(\chi - r + m_i, \chi)
\end{equation}
\[
	([\cO_{X}^{\chi} \stackrel{\varphi}{\to} F \to 0], \{V_{i} \subset F|_{x_{i}}\})
	\mapsto (\wedge^{r}\varphi, \{\psi_{i}^{-1}(V_{i})\}).
\]
In \cite[Section 7]{Tha96}, it was shown that this morphism is indeed an embedding. In \cite{Bho89}, Bhosle calculated an explicit linearization $A(\ba)$, depending on $\ba$, which gives $\widetilde{R}^{ss}(A(\ba))/\SL_{\chi} \cong \rM(r, L, \ba)$. 

In summary, the wall-crossing of $\ba$-stability is obtained by the variation of GIT. 

\subsection{Mori's program}

The wall-crossing picture can be incorporated with projective birational geometry of $\rM(r, L, \bm, \ba)$ in the nicest way. Let $\ba$ be a general parabolic weight. Then every rational contraction of $\rM(r, L, \bm, \ba)$ can be obtained in terms of wall-crossings, forgetful maps, and generalized Hecke correspondences. Proposition \ref{prop:weightanddivisor} can be recovered from \cite[Section 5]{MY20}, but for the readers' convenience, we give the proof here. 

\begin{proposition}\label{prop:weightanddivisor}
Let $\ba \in (0, 1)^{k}$ be a general parabolic weight. Then there is a linear isomorphism between a cone over $[0, 1]^{k}$ and the effective cone $\Eff(\rM(r, L, \bm, \ba))$ of divisors. 
\end{proposition}

\begin{proof}
By the GIT construction of $\rM(r, L, \bm, \ba)$ as an $\SL_{\chi}$-quotient in Section \ref{ssec:GITconstruction}, all of them can be constructed as a GIT quotient of the same smooth variety $\widetilde{R}$ with various linearizations. Furthermore,  the parabolic weights depend linearly on the choice of linearization. In particular, there is a linear embedding $(0, 1)^{k} \to \rN^{1, \SL_{\chi}}(\widetilde{R})_{\RR}$, where $\rN^{1, \SL_{\chi}}(\widetilde{R})$ is the space of numerical classes of $\SL_{\chi}$-linearized line bundles on $\widetilde{R}$. Since the character group of $\SL_{\chi}$ is trivial and $\widetilde{R}$ is normal, $\rN^{1, \SL_{\chi}}(\widetilde{R})_{\RR} \to \rN^{1}(\widetilde{R})_{\RR}$ is bijective (\cite[Corollary 1.6]{MFK94}). Applying Kempf's descent lemma (\cite[Theorem 2.3]{DN89}), we have a surjective linear map $\rN^{1}(\widetilde{R})_{\RR} \cong \rN^{1, \SL_{\chi}}(\widetilde{R})_{\RR} \to \rN^1(\widetilde{R}\git_{L}\SL_{\chi})_{\RR} = \rN^{1}(\rM(r, L, \bm, \ba))_{\RR}$. This map is bijective because the unstable locus has codimension $\ge 2$ (Theorem \ref{thm:codimunstable}). In summary, there is a linear embedding $(0, 1)^{k} \to \rN^{1, \SL_{\chi}}(\widetilde{R})_{\RR} \to \rN^1(\rM(r, L, \bm, \ba))_{\RR}$, which induces a linear embedding of a cone over $[0,1]^k$ to $\rN^1(\rM(r, L, \bm, \ba))_{\RR}$.

Now we show that the cone over the closure $[0, 1]^{k}$ of $(0, 1)^{k}$ can be identified with $\Eff(\rM(r, L, \bm, \ba))$. Recall that for any effective divisor $D$ (or equivalently, a line bundle $\cO(D)$) of a normal $\QQ$-factorial projective variety $V$, we may associate a rational contraction $V \dashrightarrow V(D)$ where 
\[
	V(D) := \proj \bigoplus_{m \ge 0}\rH^{0}(V, \cO(mD)).
\]
Conversely, any rational contraction of $V$ can be obtained in this way. If $D \in \mathrm{int}\;\Eff(V)$, then $V \dashrightarrow V(D)$ is a birational map and if $D \in \partial \Eff(V)$, $V \dashrightarrow V(D)$ is a contraction with positive dimensional general fibers. 

Note that on the boundary $\partial [0, 1]^{k}$, one of the coordinates must be either zero or one. In the first case, we can obtain a rational contraction $\rM(r, L, \bm, \ba) \to \rM(r, L, \bm', \ba')$ in Example \ref{ex:forgetfulmap}. In the latter case, we have a generalized Hecke modification in Proposition \ref{prop:generalizedHeckemodification}. All of them are contractions with positive dimensional fibers, so they must be associated with divisors on the boundary of the effective cone. Since the effective cone is convex, this is sufficient to obtain the result. 
\end{proof}


\section{Nef vector bundles}\label{sec:nef}

Let $\cE$ be the normalized Poincar\'e bundle over $X \times \rM(r, L)$. Recall that for any $x \in X$, $\cE_{x}$ is the vector bundle on $\rM(r, L)$ obtained by restricting $\cE$ on $x \times \rM(r, L)$. We prove the nefness of $\cE_{x}$ and some other positivity results. A key ingredient is the birational geometry of the moduli space of parabolic bundles with one parabolic point.

\begin{theorem}\label{thm:nef}
The restricted Poincar\'e bundle $\cE_{x}$ is a strictly nef vector bundle. 
\end{theorem}

\begin{remark}\label{rmk:nef}
The case of $d = 1$ of Theorem \ref{thm:nef} is shown in \cite[Proposition 3.3]{Nar17} and \cite[Lemma 13]{BM19}. So we assume $d > 1$. When $d = 1$, the numerical computation in Lemma \ref{lem:1stwallcrossing} is still valid. But we have $\ell = 1$ and thus, $a = 1$. Therefore, the first wall-crossing is precisely the fibration $\rM(r, L, r-1, \epsilon) \to \rM(r, L(-x))$ in Proposition \ref{prop:generalizedHeckemodification}, that is, a contraction in the Hecke correspondence. In particular, there is no flip. 
\end{remark}

We obtain another strictly nef bundle immediately. 

\begin{corollary}\label{cor:nef}
The vector bundle $\cE_{x}^{*}\otimes \Theta$ is strictly nef. 
\end{corollary}

\begin{proof}
Fix a line bundle $A$ of degree $1$ on $X.$ Consider the vector bundle $\cE^* \otimes p^* A \otimes q^* \Theta$ on $X \times \rM(r,L)$, where $p : X \times \rM(r, L) \to X$ and $q : X \times \rM(r, L) \to \rM(r, L)$ are two projections.  From the isomorphism $\rM(r,L) \cong \rM(r,L^{*}) \cong \rM(r, A^{r} \otimes L^{*}),$ we see that $\cE^* \otimes p^* A\otimes q^* \Theta$ is the normalized Poincar\'e bundle on $X \times \rM(r,A^{r} \otimes L^*) \cong X \times \rM(r, L)$. The restriction of  $\cE^* \otimes p^* A\otimes q^* \Theta$ to $x \times \rM(r,L)$ is isomorphic to $\cE_x^* \otimes \Theta$. From Theorem \ref{thm:nef}, we see that $\cE_{x}^{*}\otimes \Theta$ is strictly nef. 
\end{proof}

From now on, we prove the nefness of $\cE_{x}$. By definition, we need to show that $\cO_{\PP (\cE_{x})}(1)$ is nef. Observe that $\PP (\cE_{x}) \cong \rM(r, L, r-1, \epsilon)$ for some very small $\epsilon > 0$ (Example \ref{ex:smallweight}).  We use $\rM(r, L, r-1, a)$ for the place $\rM(r, L, (r-1), (a))$. 

We explicitly analyze the first wall-crossing of the moduli space $\rM(r, L, r-1, \epsilon)$ by increasing $\epsilon \to 1$. Recall that $\ell$ is a positive integer such that $\ell d \equiv 1 \;\mbox{mod} \;r$ and $0 < \ell < r$. 

\begin{lemma}\label{lem:1stwallcrossing}
Let $a$ be the smallest parabolic weight on a wall. Then $a = 1/\ell$. Furthermore, a maximal destabilizing subbundle has rank $k\ell$ and degree $ke$ for some $k \in \ZZ$ and an integer $e$ satisfying $\ell d - re = 1$. 
\end{lemma}

\begin{proof}
Let $\Delta(s, e, n)$ be a wall. Note that $n$ is either $s$ or $s-1$. By Equation \eqref{eqn:wallduality}, exchanging $s$ by $r-s$ if necessary, we may assume that $n = s$. Then from $(e+sa)/s = (d+(r-1)a)/r$, we have $a = (sd - re)/s$. Since $(r, d) = 1$, we can find a unique positive $0 < s < r$ and $e \in \ZZ$ such that $sd - re = 1$, which is $\ell$. 

We claim that $a = (\ell d - re)/\ell = 1/\ell$ provides the first wall. Consider a wall $a' = (s'd - re')/s'$. Setting $k:=s'd - re'$, $s'd \equiv k \;\mathrm{mod}\; r$. On the other hand, $k\ell d \equiv k \;\mathrm{mod}\; r$. So if $k\ell < r$, from the invertibility of $d$ in $\ZZ/r\ZZ$, $s' = k\ell$ and $e' = ke$. Then $a' = k/k\ell = 1/\ell = a$. If $k\ell \ge r$, there is a unique positive integer $t$ such that $0 < s' = k\ell - tr < r$. Then $a' = k/s' = k/(k\ell-tr) > k/k\ell = 1/\ell$. 

This computation tells us that $\Delta(\ell, e, \ell) = \Delta(s', e', s')$ only if $k\ell < r$ and $(s', e')=(k\ell, ke)$. So we obtain the last assertion. 
\end{proof}

We have the following diagram:
\begin{equation}\label{eqn:twocontractions}
	\xymatrix{&\PP(\cE_{x}) = \rM(r, L, r-1, \epsilon) \ar[ld]_{\pi} \ar[rd]^{\pi_{-}}\\
	\rM(r, L) && \rM(r, L, r-1, 1/\ell)}
\end{equation}
The first map $\pi$ is a projective bundle and $\pi_{-}$ is a small contraction by Corollary \ref{cor:codimcenter}. And $\rho(\PP(\cE_{x})) = \rho(\rM(r, L)) + 1 = 2$. Since $\rho(\rM(r, L, r-1, 1/\ell)) < \rho(\rM(r, L, r-1, \epsilon)) = 2$, $\rho(\rM(r, L, r-1, 1/\ell)) = 1$. Let $A$ be an ample generator of $\mathrm{Pic}(\rM(r, L, r-1, 1/\ell))$. Then $\pi^{*}\Theta$ and $\pi_{-}^{*}A$ generates $\rN^{1}(\PP(\cE_{x}))_{\RR}$. 

\begin{definition}\label{def:twocurves}
Fix a general point $((E^{-}, V^{-}), (E^{+}, V^{+}))$ in the component $\rM(\ell, e, \ell, 1/\ell) \times_{\mathrm{Pic}(X)}\rM(r-\ell, d-e, r-\ell-1, 1/\ell)$ of the wall-crossing center in $\rM(r, L, r-1, 1/\ell)$. Let $C$ be a line class in the fiber $\pi_{-}^{-1}(((E^{-}, V^{-}), (E^{+}, V^{+}))) \cong \PP \Ext^{1}((E^{+}, V^{+}), (E^{-}, V^{-}))$ (Section \ref{ssec:generalthy}).
\end{definition}

\begin{lemma}\label{lem:intersection2}
The intersection number $\cO_{\PP (\cE_{x})}(1) \cdot C$ is zero. 
\end{lemma}

\begin{proof}
The image $\pi(\PP \Ext^{1}((E^{+}, V^{+}), (E^{-}, V^{-}))) = \PP \Ext^{1}(E^{+}, E^{-}) =: \PP$ parametrizes isomorphism classes of extensions, and there is an exact sequence over $X \times \PP$
\[
	0 \to p^{*}E^{-} \otimes q^{*}\cO_{\PP}(1) \to E \otimes q^{*}\cO_{\PP}(m) \to p^{*}E^{+} \to 0
\]
(\cite[Lemma 2.3]{Ram73}, \cite[Example 2.1.12]{HL10}). Here $p : X \times \PP \to X$ and $q : X \times \PP \to \PP$ are two projections. If we restrict the exact sequence to $x \times C \cong x \times \PP^{1} \subset X \times \PP$, we have 
\[
	0 \to E^{-}_{x} \otimes \cO_{\PP^{1}}(1) \to E_{x} \otimes \cO_{\PP^{1}}(m) \to E^{+}_{x} \to 0.
\]
Since $\cE_{x}$ (and hence its restriction $E_{x}$) is normalized as $c_{1}(\cE_{x}) = \Theta^{\ell}$ where $0 < \ell < r$, and $E^{-}_{x}$ and $E^{+}_{x}$ are constant, $\ell = c_{1}(E^{-}_{x} \otimes \cO_{\PP^{1}}(1)) = c_{1}(E_{x} \otimes \cO_{\PP^{1}}(m)) = \ell + rm$. Thus, we have $m = 0$. Then $E_{x}|_{\pi(C)}$ fits in $0 \to \cO_{\PP^{1}}(1)^{\ell} \to E_{x} \to \cO_{\PP^{1}}^{r-\ell} \to 0$. A cohomology computation shows that this is a split extension. Therefore $\pi^{-1}(\pi(C)) = \PP(\cO_{\PP^{1}}(1)^{\ell} \oplus \cO_{\PP^{1}}^{r-\ell})$. The parabolic flag in $E_{x}$ is determined by that of $E^{+}_{x}$ and it is fixed over $C$. This implies that $C \cong \PP (\cO_{\PP^{1}}) \hookrightarrow \PP(\cO_{\PP^{1}}(1)^{\ell} \oplus \cO_{\PP^{1}}^{r-\ell})$. Therefore $\cO_{\PP(\cE_{x})}(1)|_{C} = \cO_{\PP (\cO_{\PP^{1}})}(1) = \cO_{\PP^{1}}$ and $\cO_{\PP (\cE_{x})}(1) \cdot C = 0$. 
\end{proof}

\begin{proof}[Proof of Theorem \ref{thm:nef}]
From $\rho(\PP (\cE_{x})) = 2$, $\pi_{-}^{*}A \cdot C = 0$, and Lemma \ref{lem:intersection2}, we can conclude that $\cO_{\PP (\cE_{x})}(1)$ and $\pi_{-}^{*}A$ are proportional. $\cO_{\PP(\cE_{x})}(1)$ is a positive multiple of $\pi_{-}^{*}A$ because it intersects with the line class in a fiber of $\pi : \PP (\cE_{x}) \to \rM(r, L)$ positively. Therefore $\cO_{\PP(\cE_{x})}(1)$ is semi-ample, so $\cO_{\PP(\cE_{x})}(1)$ and $\cE_{x}$ are nef. $\cE_{x}$ is strictly nef because $\cO_{\PP(\cE_{x})}(1)$ is not ample.
\end{proof}

The following is essentially the same computation with \cite[Proposition 3.1]{Nar17}. 

\begin{lemma}\label{lem:DetE}
Let $\cE$ be the normalized Poincar\'e bundle on $X \times \rM(r, L)$. Then 
\[
	\Det(\cE^{*}) := \det(Rq_{*}(\cE^{*}))^{-1} \cong \Theta^{\ell(1-g) - e}.
\]
\end{lemma}

\begin{proof}
For the notational simplicity, let $\rM := \rM(r, L)$ and $\rM' := \rM(r, L^{*})$. Then there is an isomorphism $\psi : \rM \to \rM'$. Since the isomorphism maps the unique ample generator $\Theta_{\rM'}$ to $\Theta_{\rM}$, by \cite[Proposition 2.1]{Nar17}, 
\[
	\Theta_{\rM} = \psi^{*}(\Theta_{\rM'}) = (\Det(\cE^{*}))^{r} \otimes (\det(\cE^{*}|_{\{x\} \times \rM}))^{-d+r(1-g)}
	= \Det(\cE^{*})^{r} \otimes \Theta_{\rM}^{-\ell(-d+r(1-g))}.
\]
Thus, $\Det(\cE^{*}) = \Theta_{\rM}^{\frac{1 + \ell (-d + r(1-g))}{r}} = \Theta_{\rM}^{-e + \ell(1-g)}$.
\end{proof}

\begin{remark}\label{rmk:suppressingpullback}
Once we fix the parabolic points and the multiplicity, $\rM(r, L, \bm, \ba)$ are all birational, and for any general $\ba$ and $\ba'$, $\rM(r, L, \bm, \ba)$ and $\rM(r, L, \bm, \ba')$ are connected by finitely many flips (Section \ref{sec:wallcrossing}). In particular, their Picard groups are identified. For a notational simplicity, we will suppress all pull-backs (by flips and regular contractions) of line bundles in our notation. For instance, when there is only one parabolic point $x$, there are two rational contractions $\pi : \rM(r, L, r-1, \epsilon) \to \rM(r, L)$ and $\pi_{1} : \rM(r, L, r-1, \epsilon) \dashrightarrow \rM(r, L, r-1, 1-\epsilon) \to \rM(r, L(-x))$. If there is no chance of confusion, we use $A \otimes B$ for $\pi^{*}A \otimes \pi_{1}^{*}B$. We denote $\cO_{\PP (\cE_{x})}(a)$ by $\cO(a)$. 
Later, when there are two parabolic points, we will set $\cO(a, b) := p_{1}^{*}\cO_{\PP(\cE_{x})}(a) \otimes p_{2}^{*}\cO_{\PP(\cE_{y}^{*})}(b)$ where $p_{1} : \PP(\cE_{x})\times_{\rM(r, L)}\PP(\cE_{y}^{*}) \to \PP(\cE_{x})$ and $p_{2} : \PP(\cE_{x})\times_{\rM(r, L)}\PP(\cE_{y}^{*}) \to \PP(\cE_{y}^{*})$.
\end{remark}

\begin{lemma}\label{lem:pullbackofample}
Let $k = (r, d-1)$. On $\rM(r, L, r-1, a)$, $\Theta_{\rM(r, L(-x))}^{k} = \cO_{\PP(\cE_{x})}(r) \otimes \Theta_{\rM(r, L)}^{1-\ell}$.
\end{lemma}

\begin{proof}
The proof is a careful refinement of \cite[Proposition 3.3]{Nar17}. We may assume that $a$ is sufficiently small, so $\rM(r, L, r-1, a) \cong \PP(\cE_{x})$. 

Let $p : X \times \PP(\cE_{x}) \to X$ and $q : X \times \PP(\cE_{x}) \to \PP(\cE_{x})$ be two projections and $\pi : X \times \PP(\cE_{x}) \to X \times \rM(r, L)$. Let $i_{x} : \PP (\cE_{x}) \cong x \times \PP(\cE_{x}) \hookrightarrow X \times \PP(\cE_{x})$. Recall that there are two exact sequences that appear on the construction of the Hecke correspondence:
\[
	0 \to H(\cE) \to \pi^{\#}(\cE) \to p^{*}\cO_{x} \otimes q^{*}\cO_{\PP (\cE_{x})}(1) \to 0
\]
and 
\begin{equation}\label{eqn:sesforpiE}
	0 \to \pi^{\#}(\cE^{*}) \to K(\cE) \to i_{x *}(\cO_{\PP(\cE_{x})}(-1)\otimes T_{x}) \to 0.
\end{equation}
Here $\pi^{\#}\cE$ is the pull-back of $\cE$ to $X \times \PP(\cE_{x})$ and $T_{x}$ is the tangent space of $X$ at $x$. 

By \cite[Proposition 2.1]{Nar17}, 
\[
	\Theta_{\rM(r, L(-x))}^{k} = \Theta_{\rM(r, L^{*}(x))}^{k} = \mathrm{Det}(K(\cE))^{r} \otimes (\det K(\cE)|_{z \times \PP(\cE_{x})})^{1-d+r(1-g)}
\]
for any $z \in X$. From \eqref{eqn:sesforpiE}, we have $\mathrm{Det}(\pi^{\#}(\cE^{*})) \otimes \cO_{\PP(\cE_{x})}(1) = \mathrm{Det}(K(\cE))$. Since $\mathrm{Det}(\pi^{\#}(\cE^{*})) = \pi^{\#}\mathrm{Det}(\cE)$ and $\pi^{\#}(\cE^{*})|_{z \times \PP(\cE_{x})} \cong K(\cE)|_{z \times \PP(\cE_{x})}$ for any $z \ne x$, 
\begin{equation}
\begin{split}
	&\mathrm{Det}(K(\cE))^{r} \otimes (\det K(\cE)|_{z \times \PP(\cE_{x})})^{1-d+r(1-g))}\\ 
	&= \mathrm{Det}(K(\cE))^{r}\otimes (\det \pi^{\#}(\cE^{*})|_{z \times \PP(\cE_{x})})^{1-d+r(1-g)}
	= \mathrm{Det}(K(\cE))^{r} \otimes \Theta_{\rM(r, L)}^{-\ell(1-d+r(1-g))}\\
	&= \pi^{\#}(\mathrm{Det}(\cE^{*}))^{r} \otimes \cO_{\PP(\cE_{x})}(r) \otimes \Theta_{\rM(r, L)}^{-\ell(1-d+r(1-g))}\\
	&= \Theta_{\rM(r, L)}^{r\ell(1-g)-re} \otimes \cO_{\PP(\cE_{x})}(r) \otimes \Theta_{\rM(r, L)}^{-\ell(1-d+r(1-g))}
	= \cO_{\PP (\cE_{x})}(r) \otimes \Theta_{\rM(r, L)}^{1-\ell}.
\end{split}
\end{equation}
The second and the fourth equalities follow from the normalization of $\cE$ and Lemma \ref{lem:DetE}, respectively. 
\end{proof}

\begin{corollary}\label{cor:pullbackofample}
Let $k = (r, d-(r-1))$. Then $\Theta_{\rM(r, L(-(r-1)y))}^{k} = \cO_{\PP(\cE_{y}^{*})}(r) \otimes \Theta_{\rM(r, L)}^{1+\ell}$.
\end{corollary}

\begin{proof}
Within the identification $\rM(r, L) \cong \rM(r, L^{*})$, the normalized Poincar\'e bundle over $\rM(r, L^{*})$ is $\cE^{*}\otimes \Theta_{\rM(r, L)}$, and $c_{1}(\cE^{*}\otimes \Theta_{\rM(r, L)}) = c_{1}(\Theta_{\rM(r, L)}^{r-\ell})$. So $\rM(r, L, 1, \epsilon) \cong \rM(r, L^{*}, r-1, \epsilon) \cong \PP(\cE_{y}^{*} \otimes \Theta_{\rM(r, L)})$. When $a \to 1$, we obtain a contraction $\rM(r, L^{*}, r-1, a) \to \rM(r, L^{*}(-y)) \cong \rM(r, L(y)) \cong \rM(r, L(-(r-1)y))$. By Lemma \ref{lem:pullbackofample}, 
\[
	\Theta_{\rM(r, L(-(r-1)y))}^{k} = \Theta_{\rM(r, L^{*})}^{1-(r-\ell)} \otimes\cO_{\PP (\cE_{y}^{*} \otimes \Theta)}(r) = \Theta_{\rM(r, L)}^{1 - (r-\ell)} \otimes \cO_{\PP(\cE_{y}^{*})}(r) \otimes \Theta_{\rM(r, L)}^{r} = \cO_{\PP(\cE_{y}^{*})}(r) \otimes \Theta_{\rM(r, L)}^{1+\ell}.
\]
\end{proof}

From \eqref{eqn:twocontractions}, we obtain the nef cones of $\PP(\cE_{x})$ and $\PP(\cE_{x}^{*})$. The bigness in the statement follows from Lemma \ref{lem:pullbackofample} and Corollary \ref{cor:pullbackofample}. 

\begin{corollary}\label{cor:nefcone}
\begin{enumerate}
\item The nef cone of $\PP(\cE_{x}) = \rM(r, L, r-1, \epsilon)$ is generated by $\pi^{*}\Theta$ and $\cO_{\PP(\cE_{x})}(1)$. If $d \ne 1$, $\cO_{\PP(\cE_{x})}(1)$ is big. 
\item The nef cone of $\PP(\cE_{x}^{*}) = \rM(r, L, 1, \epsilon)$ is generated by $\pi^{*}\Theta$ and $\cO_{\PP(\cE_{x}^{*})}(1)\otimes \pi^{*}\Theta$. If $d \ne r-1$, $\cO_{\PP(\cE_{x}^{*})}(1) \otimes \pi^{*}\Theta$ is big. 
\end{enumerate}
\end{corollary}


\section{Main example}\label{sec:mainexample}

From now on, we focus on the case that $k = 2$ and $\bm = (r-1, 1)$. We set $\bx = (x_1, x_2)$ and $\ba = (a_{1}, a_{2})$. We use $\rM(r, L, \ba)$ for $\rM(r, L, \bm, \ba)$.

\subsection{Effective cone}\label{ssec:effectivecone}

Let $\Delta(s, e, \bn)$ be a wall on $[0,1]^{2}$ and let $\ba$ be a general point on it. Let $(E, V_{\bullet}) \in Y \subset \rM(r, L, \ba)$ be a general polystable parabolic bundle on the wall-crossing center. Then $(E, V_{\bullet}) \cong (F_{1}, W_{1 \bullet}) \oplus (F_{2}, W_{2\bullet})$ and $\mu(E, V_{\bullet}) =\mu(F_{1}, W_{1\bullet}) = \mu(F_{2}, W_{2\bullet})$.

There are two possibilities. First of all, it is possible that one of $F_{i}$'s (say $F_{1}$) has the largest possible intersection with the flags of $E$. That means, $\dim F_{1}|_{x_1} \cap V_{1} = \dim F_{1}|_{x_1} = s$ and $\dim F_{1}|_{x_2} \cap V_{2} = \dim V_{2} = 1$. We have an equality 
\[
	\frac{e+sa_{1} + a_{2}}{s} = \frac{d + (r-1)a_{1} + a_{2}}{r},
\]
or equivalently, $sa_{1} + (r-s)a_{2} = sd - re$. The slope of the line on the $(a_{1}, a_{2})$-plane is negative, so we will call the wall a \emph{negative wall}. To intersect with the interior of $[0, 1]^{2}$, it is necessary that $0 < sd - re < r$. Since these walls are $\Delta(s, e, (s, 1)) = \Delta(r-s, d-e, (r-s-1, 0))$, they are simple walls (Section \ref{ssec:generalthy}).

The other case is that $\dim F_{1}|_{x_1} \cap V_{1} = \dim F_{1}|_{x_1} = s$ and $\dim F_{1}|_{x_2} \cap V_{2} = 0$. Then 
\[
	\frac{e+sa_{1}}{s} = \frac{d+(r-1)a_{1} + a_{2}}{r}, 
\]
so $sa_{1} - sa_{2} = sd - re$. The slope of the wall $\Delta(s, e, (s, 0))$ is one and we call it a \emph{positive wall}. The nonempty intersection with $(0, 1)^{2}$ is equivalent to $-s < sd - re < s$. Since $(r, d) = 1$, $sd - re \ne 0$ and there is no wall passing through the origin. See Figure \ref{fig:wallchamber} for an example of the wall-chamber decomposition.

\begin{figure}[!ht]
\begin{tikzpicture}[scale=7]
	\draw[line width = 1pt] (0, 0) -- (1, 0);
	\draw[line width = 1pt] (1, 0) -- (1, 1);
	\draw[line width = 1pt] (0, 0) -- (0, 1);
	\draw[line width = 1pt] (0, 1) -- (1, 1);
	\draw[line width = 1pt, color=cyan, -latex] (0, 0) -- (1, 1);
	\draw[line width = 1pt, color=cyan, -latex] (0, 0) -- (1, 5/6);
	\draw[line width = 1pt] (0.33, 0) -- (0, 0.5);
	\draw[line width = 1pt] (1, 0.25) -- (0, 0.5);
	\draw[line width = 1pt] (0.75, 0) -- (0.5, 1);
	\draw[line width = 1pt] (1, 0.66) -- (0.5, 1);
	\draw[line width = 1pt] (0.75, 0) -- (1, 0.25);
	\draw[line width = 2pt] (0, 0.5) -- (0.5, 1);
	\draw[line width = 1pt] (0.33, 0) -- (1, 0.66);

	\node at (-0.05, -0.05) {$0$};
	\node at (1.05, -0.05) {$1$};
	\node at (-0.05, 1.05) {$1$};
	\node at (-0.05, 0.5) {$\frac{1}{2}$};
	\node at (0.33, -0.05) {$\frac{1}{3}$};
	\node at (0.75, -0.05) {$\frac{3}{4}$};
	\node at (1.05, 0.25) {$\frac{1}{4}$};
	\node at (1.05, 0.66) {$\frac{2}{3}$};
	\node at (0.5, 1.05) {$\frac{1}{2}$};
	\node at (0.2, 0.2) {$\Delta(3, 1, (3, 1))$};
	\node at (0.3, 0.5) {$\Delta(1, 0, (1, 1))$};
	\node at (0.55, 0.65) {$\Delta(4, 1, (4, 1))$};
	\node at (0.8, 0.8) {$\Delta(2, 0, (2, 1))$};
	\node at (0.15, 0.75) {$\Delta(2, 1, (2, 0)) = \Delta(4, 2, (4, 0))$};
	\node at (0.85, 0.5) {$\Delta(3, 1, (3, 0))$};
	\node at (0.9, 0.15) {$\Delta(4, 1, (4, 0))$};
\end{tikzpicture}
\caption{The wall-chamber decomposition for $r = 5$ and $d = 2$. The thick line segment for $\Delta(2, 1, (2, 0))$ is a multiple wall. Two arrows denote diagonal and nearly diagonal wall crossing directions.}\label{fig:wallchamber}
\end{figure}
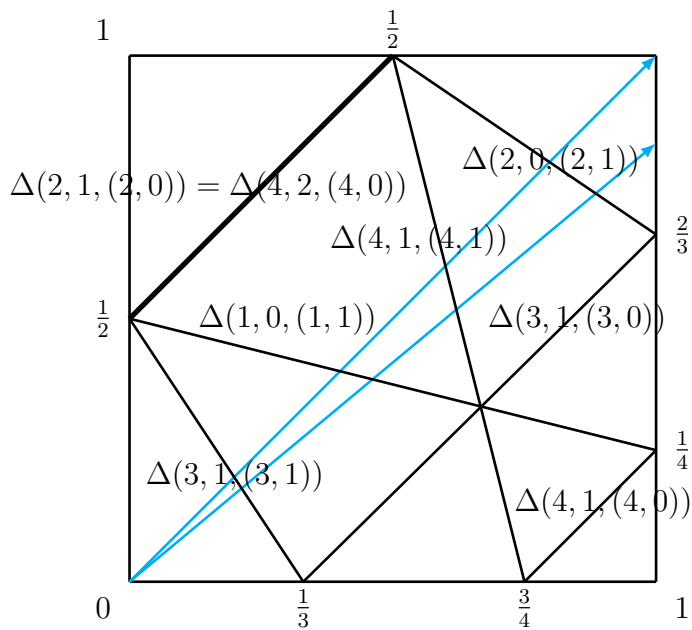

The line bundle $\Theta$ is the pull-back of $\Theta$ by $\rM(r, L, \ba) \dashrightarrow \rM(r, L)$. 

\begin{lemma}\label{lem:canonicaldivisor}
For a general weight $\ba$, the dualizing bundle of $\rM(r, L, \ba)$ is
\[
	\omega = \cO(-r, -r) \otimes \Theta^{-2}.
\]
\end{lemma}

\begin{proof}
We may assume that $\ba$ is sufficiently small and $\rM(r, L, \ba) \cong \PP(\cE_{x_1}) \times_{\rM(r, L)}\PP(\cE_{x_2}^{*})$. Apply the relative Euler sequence to $\PP(\cE_{x_1}) \to \rM(r, L)$ and $\PP(\cE_{x_1}) \times_{\rM(r, L)}\PP(\cE_{x_2}^{*}) \to \PP(\cE_{x_1})$. 
\end{proof}

\begin{proposition}\label{prop:effectivecone}
Let $\ba$ be a general weight. Then $\Eff(\rM(r, L, \ba))$ is generated by four extremal rays
\[
	\Theta, \cO(r, 0) \otimes \Theta^{1-\ell}, \cO(0, r) \otimes \Theta^{1+\ell}, \cO(r, r) \otimes \Theta.
\]
\end{proposition}
\begin{proof}
By Proposition \ref{prop:weightanddivisor}, it is sufficient to find four divisors associated to four extremal parabolic weights. For any big $\QQ$-divisor $D \in \mathrm{int}\;\Eff(\rM(r, L, \ba))$, the associated birational model 
\[
	\rM(r, L, \ba)(D) := \proj \bigoplus_{m \ge 0}\rH^{0}(\rM(r, L, \ba), \cO(\lfloor mD\rfloor))
\] 
is $\rM(r, L, \ba')$. When $\ba' = (0, 0)$, the associated rational contraction is $\rM(r, L)$ and the associated divisor is a scalar multiple of $\Theta$. When $\ba' = (1/\ell, 0)$, by Section \ref{sec:nef}, the associated divisor is a multiple of $\cO(1, 0)$. When $\ba' = (1, 0)$, the associated rational contraction is $\rM(r, L(-x))$ and the associated divisor is a scalar multiple of $\cO(r, 0) \otimes \Theta^{1-\ell}$ by Lemma \ref{lem:pullbackofample}. For $\ba' = (0, 1/(r-\ell))$, we have a multiple of $\cO(0, 1) \otimes \Theta$. Finally, for $\ba' = (0, 1)$, a multiple of $\cO(0, r) \otimes \Theta^{1+\ell}$ is associated. 

By an elementary computation, for each point $\ba' = (a_1', a_2') \in [0, 1]^{2}$, the associated divisor $\cO(D) = \cO(c_1, c_2) \otimes \Theta^d$ can be written as a positive multiple of $\Theta \otimes (\cO(r, 0) \otimes \Theta^{-\ell})^{a_{1}'} \otimes (\cO(0, r) \otimes \Theta^{\ell})^{a_{2}'} = \cO(ra_{1}', ra_{2}') \otimes \Theta^{1 - \ell a_{1}' + \ell a_{2}'}$. A routine calculation shows that 
\begin{equation}\label{eqn:weighttobundle}
	(a_{1}', a_{2}') = \left(\frac{c_{1}}{rd + \ell c_{1} - \ell c_{2}}, \frac{c_{2}}{rd + \ell c_{1} - \ell c_{2}}\right).
\end{equation}
Thus, the last extremal ray, which is associated to $\ba' = (1, 1)$, is $\cO(r, r) \otimes \Theta$. 
\end{proof}


\subsection{Diagonal and nearly diagonal wall-crossings}\label{ssec:diagonal}

We say a wall crossing is a \emph{diagonal} one if we cross a wall $\Delta(s, e, \bn)$ while the weight $\ba$ is increasing along the line $a_{1} = a_{2}$. A wall crossing is a \emph{nearly diagonal} if we cross a wall $\Delta(s, e, \bn)$ while the weight $\ba$ is increasing along $ra_{1} = (r+1)a_{2}$. See Figure \ref{fig:wallchamber}. We explicitly compute these wall-crossings. 

\begin{proposition}\label{prop:diagonalissimple}
All walls that appear in diagonal or nearly diagonal wall crossings are simple. 
\end{proposition}

\begin{proof}
For a negative wall $\Delta(s, e, (s, 1)) = \Delta(r-s, d-e, (r-s-1, 0))$, the greatest common divisor for both $\{s, e, s, 1\}$ and $\{r-s, d-e, r-s-1, 0\}$ are one. So every negative wall is simple. Then all multiple walls are positive walls, and hence parallel to the diagonal line $a_{1} = a_{2}$. Such a wall is given by $a_{1} - a_{2} = (sd - re)/s$. Since $(r, d) = 1$, the right hand side is nonzero and it is disjoint from the diagonal line $a_{1} = a_{2}$. Moreover, $|(sd - re)/s| \ge 1/s \ge 1/(r-1)$. It is a routine calculation to check that these walls do not intersect with $ra_{1} = (r+1)a_{2}$ on $[0, 1]^{2}$.
\end{proof}

\begin{remark}
Several walls can meet at a weight during diagonal or nearly diagonal wall crossings. In this case, we may perturb the weight slightly, then the wall-crossing can be decomposed into a composition of several simple wall-crossings. Thus, we may assume that all wall-crossings are simple.
\end{remark}

We can compute the dimension of all simple wall-crossing centers $Y_{\pm}$. For the theoretical background and details, see \cite[Section 4]{MY21}. Here we leave the computation for a negative wall $\Delta(s, e, (s, 1))$. We keep the notation in the diagram \eqref{eqn:wallcrossing}. 

For a point $((E^{+}, V_{\bullet}^{+}), (E^{-}, V_{\bullet}^{-})) \in Y_{0}$, it is sufficient to evaluate $\dim \PP \Ext^{1}((E^{\pm}, V_{\bullet}^{\pm}), (E^{\mp}, V_{\bullet}^{\mp}))$. By the Serre duality for parabolic bundles, 
\[
	\Ext^{1}((E^{-}, V_{\bullet}^{-}), (E^{+}, V_{\bullet}^{+})) \cong \spHom((E^{+}\otimes \omega^{*}(-\bx), V_{\bullet}^{+}), (E^{-}, V_{\bullet}^{-}))^{*}
\]
(\cite[Proposition 3.7]{Yok95}). There is an exact sequence of vector spaces (\cite[Section 4.2]{MY21})
\begin{equation}\label{eqn:les}
\begin{split}
	0 & \to \spHom((E^{+} \otimes \omega^{*}(-\bx), V_{\bullet}^{+}), (E^{-}, V_{\bullet}^{-})) \to \rHom(E^{+}\otimes \omega^{*}(-\bx), E^{-}) \\
	& \to \bigoplus_{i=1}^{2}\rHom(E^{+}\otimes \omega^{*}(-\bx)|_{x_{i}}, E^{-}|_{x_{i}})/\mathrm{N}_{x_{i}}((E^{+}\otimes \omega^{*}(-\bx), V_{\bullet}^{+}), (E^{-}, V_{\bullet}^{-})) \to 0,
\end{split}
\end{equation}
where $N_{x}((E^{+}, V_{\bullet}^{+}), (E^{-}, V_{\bullet}^{-}))$ is the subspace of $\rHom(E^{+}|_{x}, E^{-}|_{x})$ which is strongly parabolic at $x$. Since the parabolic weight for $V_{i}^{+}$ and $V_{i}^{-}$ are the same, 
\[
\begin{split}
	&N_{x_{i}}((E^{+}\otimes \omega^{*}(-\bx), V_{\bullet}^{+}), (E^{-}, V_{\bullet}^{-}))\\
	& = \{f \in \rHom(E^{+}\otimes \omega^{*}(-\bx)|_{x_{i}}, E^{-}|_{x_{i}})\;|\; f(E^{+}\otimes \omega^{*}(-\bx)|_{x_{i}}) \subset V_{i}^{-}, f(V_{i}^{+}) = 0\}.
\end{split}
\]
From $\dim V_{1}^{+}= s = \rk E^{+}$ and $\dim V_{2}^{-} = 0$, $\dim N_{x_{i}}((E^{+}\otimes \omega^{*}(-\bx), V_{\bullet}^{+}), (E^{-}, V_{\bullet}^{-})) = 0$ for both $i=1,2$. Now 
\begin{equation}\label{eqn:inequality1}
\begin{split}
	& \dim \Ext^{1}((E^{-}, V_{\bullet}^{-}), (E^{+}, V_{\bullet}^{+}))\\ &= \dim \rHom(E^{+}\otimes \omega^{*}(-\bx), E^{-}) - 2s(r-s) \\
	& \ge \chi(E^{+ *} \otimes E^{-} \otimes \omega(\bx)) - 2s(r-s) = sd - re + s(r-s)(g-1).
\end{split}
\end{equation}
By the same method, we obtain 
\begin{equation}\label{eqn:inequality2}
	\dim \Ext^{1}((E^{+}, V_{\bullet}^{+}), (E^{-}, V_{\bullet}^{-})) \ge 
	re - sd + s(r-s)(g-1) + r,  
\end{equation}
so $\dim \Ext^{1}((E^{-}, V_{\bullet}^{-}), (E^{+}, V_{\bullet}^{+})) + \dim \Ext^{1}((E^{+}, V_{\bullet}^{+}), (E^{-}, V_{\bullet}^{-})) \ge 2s(r-s)(g-1) + r$.
On the other hand, 
\[
\begin{split}
	&\dim \PP \Ext^{1}((E^{-}, V_{\bullet}^{-}), (E^{+}, V_{\bullet}^{+})) + \dim \PP \Ext^{1}((E^{+}, V_{\bullet}^{+}), (E^{-}, V_{\bullet}^{-}))\\
	& = \dim \rM(r, L, \ba) - \dim \rM(s, e, \bn, \ba) \times_{\Pic^{e}(X)}\rM(r-s, d-e, \bm - \bn, \ba) - 1\\
	& = 2s(r-s)(g-1) + r -2.
\end{split}
\]
Therefore, we obtain that \eqref{eqn:inequality1} and \eqref{eqn:inequality2} are indeed equalities. In summary:

\begin{proposition}\label{prop:fiberdim}
Let $\Delta(s, e, (s, 1))$ be a negative wall. For the contraction map $\pi_{\pm} : \rM(r, L, \ba^{\pm}) \to \rM(r, L, \ba)$ in \eqref{eqn:wallcrossing}, the dimension of the exceptional fiber of $\pi_{+}$ (resp. $\pi_{-}$) is $(re-sd) + s(r-s)(g-1) + r-1$ (resp. $(sd - re) + s(r-s)(g-1)-1$). 
\end{proposition}


\section{Cohomology via wall-crossing of derived category}\label{sec:cohomologyderivedcategory}

To prove the main theorems, a critical technical step is to identify the cohomology groups of the bundles on different birational models. Halpern-Leistner and Ballard-Favero-Katzarkov provided a systematic way to study the derived category of a variation of GIT (\cite{HL15, BFK19}). In this section, we review their works, in particular the quantization theorem. Technically, the results treat the derived category of a quotient stack. However, the following (well-known) lemma and its corollary show that it can be applied to the cohomology computation on the coarse moduli space. Let $\rD^{\mathrm{per}}(\rM)$ be the category of perfect complexes over $\rM$.

\begin{lemma}\label{lem:dbofcoarsemoduli}
Let $\cM$ be a smooth Artin stack and $\pi : \cM \to \rM$ be its good moduli space. Then $L\pi^* :\rD^{\mathrm{per}}(\rM) \to \rD^b(\cM)$ is fully faithful. 
\end{lemma}

\begin{proof}
We have an isomorphism $\cO_{\rM} \to R \pi_* \cO_{\cM}$ from the definition of a good moduli space. For any $F^{\bullet}, G^{\bullet} \in \rD^{\mathrm{per}}(\rM)$ and $i \in \ZZ$, we have isomorphisms 
\[
	\rHom(L\pi^* F^{\bullet}, L\pi^* G^{\bullet}[i]) \cong \rHom(F^{\bullet}, R\pi_* L\pi^*  G^{\bullet}[i]) \cong \rHom(F^{\bullet}, G^{\bullet}[i])
\]
by the adjunction formula and the projection formula (\cite[Corollary 4.12]{HR17}, \cite[Proposition 9.3.6]{Ols16}). Therefore we see that $L\pi^*$ is fully faithful.
\end{proof}

\begin{corollary}\label{cor:dbofcoarsemoduli}
We retain the same setup.
\begin{enumerate}
\item If $L$ is a vector bundle over $\rM$, then $\rH^{i}(\rM, L) \cong \rH^{i}(\cM, \pi^{*}L)$. 
\item If $\rM$ is smooth, $L\pi^* : \rD^b(\rM) \to \rD^b(\cM)$ is fully faithful. In particular, for any $F^{\bullet} \in \rD^{b}(\rM)$, $\rHom^{i}(\cO_{\rM}, F^{\bullet}) \cong \rHom^{i}(\cO_{\cM}, L\pi^{*}F^{\bullet})$. 
\end{enumerate}
\end{corollary}

\subsection{Variation of GIT and derived category}\label{ssec:VGITderivedcategory}

Let $V$ be a smooth quasi projective variety equipped with a reductive group $G$-action and $A$ be a linearization. The GIT quotient $V \git_{A}G$ is the good moduli space of the quotient stack $[V^{ss}(A)/G]$. Halpern-Leistner showed that, for a collection of integers $w = (w_{i})$ for each Kempf-Ness stratum of the unstable locus, $\rD^{b}([V/G])$ has a semiorthogonal decomposition 
\[
	\rD^{b}([V/G]) = \langle \rD^{b}_{[V^{us}(A)/G]}([V/G])_{<w}, \mathbf{G}_{w}, \rD^{b}_{[V^{us}(A)/G]}([V/G])_{\ge w})\rangle,
\]
and moreover, the restriction functor $i^{*} : \mathbf{G}_{w} \to \rD^{b}([V^{ss}(A)/G])$ is an equivalence of categories (\cite[Theorem 2.10]{HL15}). 

From now on, we assume that there is only one unstable stratum $S$ that is a smooth subvariety. It is determined by a one-parameter subgroup $\lambda(t)$ which minimizes the normalized weight $\wt_{\lambda}A/|\lambda|$ over the $\lambda$-fixed locus $Z \subset S$. Since a choice of $w$ is arbitrary, we may set $w = 0$. Under this condition, $\mathbf{G}_{w}$ is characterized as the subcategory of complexes $F^{\bullet}$ such that the $\lambda$-weights of the hypercohomology $\cH^{*}(F^{\bullet}|_{Z})$ is supported on $[w, w+\eta)$ (\cite[Lemma 2.9]{HL15}). Here $\eta$ is the $\lambda$-weight of the top wedge product of $N_{S/V}^{*}|_{Z}$. 

The following theorem is a key ingredient for our cohomology computation. 

\begin{theorem}[Quantization Theorem \protect{\cite[Theorem 3.29]{HL15}}]\label{thm:quantization}
For $F^{\bullet} \in \rD^{b}([V/G])$, suppose that the $\lambda$-weights of $\cH^{*}(F^{\bullet}|_{Z})$ are supported on $(-\infty, \eta)$. Then 
\[
	\rH^{i}([V/G], F^{\bullet}) \cong \rH^{i}([V^{ss}(A)/G], F^{\bullet}|_{[V^{ss}(A)/G]}).
\]
\end{theorem}

We apply the above result to the variation of GIT setup. Let $A_{0}$ be a linearization such that $V^{ss}(A_{0}) \ne V^{s}(A_{0})$. For a sufficiently small $\epsilon$ and a linearized ample line bundle $A$, let $A_{\pm} := A_{0} \pm \epsilon A$. We assume that $V^{ss}(A_{\pm}) = V^{s}(A_{\pm})$. Assume further that $V^{s}(A_{\pm}) = V^{ss}(A_{0}) \setminus S_{\pm}$ and $S_{\pm}$ are smooth irreducible varieties. If $\lambda_{\pm}$ are the one-parameter subgroups describing the Kempf-Ness strata $S_{\pm}$, then $\lambda_{-} = \lambda_{+}^{-1}$. Let $Z \subset S_{+} \cap S_{-}$ be the $\lambda_{\pm}$-fixed locus, and $\eta_{\pm}$ be the $\lambda_{\pm}$-weight of the top wedge product of $N_{S_{\pm}/V}^{*}|_{Z}$. 

\begin{theorem}[\protect{\cite[Theorem 3.15]{TT21}}]\label{thm:cohomologyidentification}
If $\lambda_{-}$-weights of $\cH^{*}(F^{\bullet}|_{Z})$ are supported on $(-\eta_{+}, \eta_{-})$,
\[
	\rH^{i}([V^{s}(A_{-})/G], F^{\bullet}|_{[V^{s}(A_{-})/G]}) \cong \rH^{i}([V^{ss}(A_{0})/G], F^{\bullet}) \cong \rH^{i}([V^{s}(A_{+})/G], F^{\bullet}|_{[V^{s}(A_{+})/G]}).
\]
\end{theorem}

The first isomorphism follows from Theorem \ref{thm:quantization} and the second one is from the theorem and $\lambda_{-} = \lambda_{+}^{-1}$. In particular, for any line bundle $E$ on $[V^{ss}(A_{0})/G]$, if the magnitude of the $\lambda_{-}$-weight is `not too big,' then the cohomology of $E$ on both sides of the wall can be identified.

\subsection{Weight computation}\label{ssec:weight}

All $\rM(r, L, \ba)$ are constructed by GIT and they are connected by the variation of GIT (Section \ref{ssec:GITconstruction}). For any simple wall-crossing, the technical assumptions we made in Section \ref{ssec:VGITderivedcategory} hold. In this section, we compute the $\lambda_{-}$-weight $\mathrm{wt}_{\lambda_{-}}F$ for every line bundle $F$ and each simple wall that occurs during the diagonal and nearly diagonal wall-crossings. 

Take a wall $\Delta(s, e, \bn)$ and pick a general weight $\ba = (a_{1}, a_{2}) \in \Delta(s, e, \bn)$. Let $A$ be an ample divisor associated to $\ba$. Then $\widetilde{R}^{ss}(A)\git \SL_{\chi} \cong \rM(r, L, \ba)$. For two nearby weights $\ba_{\pm}:= (a \pm \epsilon, a \pm \epsilon)$, let $A_{\pm}$ be a line bundle such that $\rM(r, L, \ba_{\pm}) \cong \widetilde{R}^{s}(A_{\pm})\git \SL_{\chi}$. 

\begin{proposition}\label{prop:thetalambdaweight}
Let $\Delta(s, e,\bn)$ be a simple wall, $\ba \in \Delta(s, e, \bn)$, and $A$ be an associated line bundle. Let $\lambda_{-}$ be the one-parameter subgroup associated to the stratum $S_{-} := \widetilde{R}^{ss}(L_{0}) \setminus \widetilde{R}^{s}(L_{-})$. Over the $\lambda_{-}$-fixed locus $Z \subset S_{-}$,
\[
	\mathrm{wt}_{\lambda_{-}}\Theta = -\chi(sd - re).
\]
\end{proposition}

\begin{proof}
A general point in $Z$ parametrizes a pair 
\[
	([\cO^{\chi^{+}}\oplus \cO^{\chi^{-}} \stackrel{\varphi}{\to} E^{+}(m) \oplus E^{-}(m) \to 0], V_{\bullet}),
\]
where $E^{+}$ (resp. $E^{-}$) is a rank $s$ (resp. $r-s$), degree $e$ (resp. $d-e$) vector bundle, $\chi^{\pm} = \dim \rH^{0}(E^{\pm}(m))$, and $\varphi = \varphi^{+} \oplus \varphi^{-}$ where $\varphi^{\pm} : \cO^{\chi^{\pm}} \to E^{\pm}(m)$ and $\rH^{0}(\cO^{\chi^{\pm}}) \to \rH^{0}(E^{\pm}(m))$ is a scalar multiple map. Because $\lambda_{-}(t)$ is a subgroup of $\SL_{\chi}$, $\lambda_{-}(t)$-weight on $E^{+}(m)$ is $-u\chi^{-}$ and that on $E^{-}(m)$ is $u\chi^{+}$ for some scalar $u$. Normalizing $\lambda_{-}$, we may assume that $u = 1$. By Riemann-Roch, it is straightforward to check that $\chi^{+} = e + sm + s(1-g)$ and $\chi^{-} = (d-e) + (r-s)m + (r-s)(1-g)$. 

For any vector bundle $E$ over $X$, $\Theta|_{[E]} = \Theta|_{[E(m)]}$ is defined as 
\[
	\mathrm{Det}(E(m))^{r} \otimes \det E(m)|_{x}^{\chi} = (\wedge^{\chi}\rH^{0}(E(m))^{*})^{r}\otimes E(m)|_{x}^{\chi}.
\]
for some $x \in X$ (\cite[Proposition 2.1]{Nar17}). Its $\lambda_{-}$-weight is 
\[
	-r(-\chi^{-}\chi^{+}+\chi^{+}\chi^{-}) + \chi(-s\chi^{-} + (r-s)\chi^{+}) = -\chi(sd - re).
\]	
\end{proof}

\begin{lemma}\label{lem:wallimplieszeroweight}
Under the same assumption, $\mathrm{wt}_{\lambda_{-}} A = 0$.
\end{lemma}

\begin{proof}
Recall that $\rM(r, L, \ba) = \widetilde{R}^{ss}(A)\git \SL_{\chi}$. Since $A$ descends to $\rM(r, L, \ba)$, by Kempf's descent lemma (\cite[Theorem 2.3]{DN89}), for any closed $\SL_{\chi}$-orbit, the stabilizer group acts on the fiber of $A$ trivially. In particular, at a point in $Z$, the stabilizer group $\lambda_{-}$ acts trivially on the fiber, hence the $\lambda_{-}$-weight is zero. 
\end{proof}

On the other hand, for a point $z := ([\cO^{\chi^{+}}\oplus \cO^{\chi^{-}} \to E^{+}(m)\oplus E^{-}(m) \to 0], V_{\bullet}) \in Z$, $N_{S_{-}/\widetilde{R}}|_{z}$ is identified with $\Ext^{1}((E^{-}, V_{\bullet}^{-}), (E^{+}, V_{\bullet}^{+}))$ and the action of $\lambda_{-}$ on $N_{S_{\pm}/\widetilde{R}^{ss}(A)}|_{z}$ has weight $-\chi$ (\cite[Section 7]{Tha96}). Thus, for a negative wall-crossing along $\Delta(s, e, (s, 1))$, we obtain
\begin{equation}\label{eqn:etam}
	\eta_{-} = \chi \dim \Ext^{1}((E^{-}, V_{\bullet}^{-}), (E^{+}, V_{\bullet}^{+})) = \chi(sd - re +s(r-s)(g-1)), 
\end{equation}
\begin{equation}\label{eqn:etap}
	\eta_{+} = \chi \dim \Ext^{1}((E^{+}, V_{\bullet}^{+}), (E^{-}, V_{\bullet}^{-})) = \chi(re - sd +s(r-s)(g-1) + r).
\end{equation}
by Proposition \ref{prop:fiberdim}.


\section{Embedding of derived category}\label{sec:embedding}

In this section, we prove Theorem \ref{thm:mainthmembedding}.

\subsection{Bondal-Orlov criterion}

Let $\cE$ be the normalized Poincar\'e bundle over $X \times \rM(r, L)$. Let $p : X \times \rM(r, L) \to X$, $q : X \times \rM(r, L) \to \rM(r, L)$ be two projections. Consider the Fourier-Mukai transform 
\begin{eqnarray*}
	\Phi_{\cE} : \rD^{b}(X) &\to& \rD^{b}(\rM(r, L))\\
	F^{\bullet} & \mapsto & Rq_{*}(\cE \otimes^L Lp^{*} F^{\bullet}).
\end{eqnarray*}

The Bondal-Orlov criterion (\cite[Theorem 1.1]{BO95}) provides the necessary and sufficient condition for the fully-faithfulness of a Fourier-Mukai transform between two smooth algebraic varieties. The next theorem is a version applied to $\Phi_{\cE}$.

\begin{theorem}[Bondal-Orlov criterion]\label{thm:vanishing}
For each $x \in X$, let $\cE_{x}$ be the restriction of the normalized Poincar\'e bundle on $\rM(r, L)$. The functor $\Phi_{\cE} : \rD^{b}(X) \to \rD^{b}(\rM(r, L))$ is fully faithful if and only if the following conditions hold:
\begin{enumerate}
\item $\rH^{0}(\rM(r, L), \cE_{x} \otimes \cE_{x}^{*}) \cong \CC$. 
\item $\rH^{i}(\rM(r, L), \cE_{x} \otimes \cE_{x}^{*}) = 0$ for $i \ge 2$. 
\item $\rH^{i}(\rM(r, L), \cE_{x_1} \otimes \cE_{x_2}^{*}) = 0$ for all $x_1 \ne x_2$ and all $i \in \ZZ$. 
\end{enumerate}
\end{theorem}

\begin{proof}[Proof of Theorem \ref{thm:vanishing}]
Items (1) and (2) are proved by \cite[Section 3]{BM19} by extending the work of Narasimhan and Ramanan in \cite{NR75}. We show Item (3). Since 
\[
	\rH^{i}(\rM(r, L, \be), \cO(1, 1)) \cong \rH^{i}(\rM(r, L), \cE_{x_{1}} \otimes \cE_{x_{2}}^{*})
\]
for a small $\be = (\epsilon, \epsilon)$, it is sufficient to show that $\rH^{i}(\rM(r, L, \be), \cO(1, 1)) = 0$. 

By Proposition \ref{prop:effectivecone} and the fact that there is no divisorial contraction on the wall-crossing (Proposition \ref{prop:fiberdim}), there is a parabolic weight $\ba$ such that $\cO(r+1, r+1) \otimes \Theta^{2}$ is nef and big on $\rM(r, L, \ba)$. Note that $\cO(r+1, r+1) \otimes \Theta^{2}$ lies on a subspace generated by two extremal rays $\Theta$ and $\cO(r, r) \otimes \Theta$ of $\Eff(\rM(r, L, \be))$. To reach this line bundle, we may run a diagonal wall-crossing. By Proposition \ref{prop:diagonalissimple}, we encounter only negative walls, which are all simple, to reach $\rM(r, L, \ba)$ from $\rM(r, L, \be)$. 

For each negative wall $\Delta(s, e, (s, 1))$, a parabolic weight $\ba' = (a_{1}', a_{2}')$ lies on it if and only if it satisfies $sa_{1}'+(r-s)a_{2}' = sd - re$. Furthermore, if $\ba'$ is on the diagonal, $a_{1}' = a_{2}' = (sd-re)/r$. Thus, by \eqref{eqn:weighttobundle} in Section \ref{ssec:effectivecone}, the associated line bundle is a scalar multiple of 
\[
	\cO(sd-re, sd-re) \otimes \Theta.
\]
The $\lambda_{-}$-weight for this line bundle has to be zero by Lemma \ref{lem:wallimplieszeroweight}. By Proposition \ref{prop:thetalambdaweight}, 
\[
	\mathrm{wt}_{\lambda_{-}}\cO(1, 1) = \chi.
\]

On the other hand, since $g \ge 2$ and $0 < sd - re < r$, Equations \eqref{eqn:etam} and \eqref{eqn:etap} tell us $\eta_{\pm} > \chi$. Therefore, for any simple wall intersecting the diagonal, the $\lambda_{-}$-weight of $\cO(1,1)$ lies on $(-\eta_{-}, \eta_{+})$. Theorem \ref{thm:cohomologyidentification} implies that 
\[
	\rH^{i}(\rM(r, L, \ba'), \cO(1, 1)) \cong \rH^{i}(\rM(r, L, \be), \cO(1, 1))
\]
for any $i \in \ZZ$ and any general diagonal weight $\ba'$, including $\ba$. For $i > 0$, 
\[
	\rH^{i}(\rM(r, L, \ba), \cO(1, 1)) = \rH^{i}(\rM(r, L, \ba), \omega \otimes \cO(r+1, r+1) \otimes \Theta^{2}) = 0
\]
by Kawamata-Viehweg vanishing. And $\rH^{0}(\rM(r, L, \ba), \cO(1, 1)) = 0$ since $\cO(1, 1) \notin \Eff(\rM(r, L, \ba))$. 
\end{proof}


\section{Vanishing of cohomology}\label{sec:cohomologyvanishing}

We prove the following vanishing result, which is used in both the computation of a semiorthogonal decomposition of $\rD^{b}(\rM(r, L))$ and the construction of ACM bundles. 

\begin{theorem}\label{thm:vanishingoftwistedbundle}
For any $x \in X$ and $j \ge -1$, $\rH^{i}(\rM(r, L), \cE_{x} \otimes \Theta^{j}) = 0$ for all $i > 0$. 
\end{theorem}

\begin{proof} We divide the proof into several steps.

\textbf{Step 1.}
Observe that 
\[
	\rH^{i}(\rM(r, L), \cE_{x} \otimes \Theta^{j}) \cong \rH^{i}(\rM(r, L, \be), \cO(1, 0)  \otimes \Theta^{j})
	\cong \rH^{i}(\rM(r, L, \be), \cO(r+1, r) \otimes \Theta^{j+2} \otimes \omega).
\]
Since 
\[
	\cO(r+1, r) \otimes \Theta^{j+2} = (\cO(r, r) \otimes \Theta) \otimes \Theta^{\frac{\ell-1}{r} + (j+1)} \otimes (\cO(r, 0) \otimes \Theta^{1-\ell})^{\frac{1}{r}}, 
\]
for $j \ge -1$, $\cO(r+1, r) \otimes \Theta^{j+2}$ is on the effective cone of $\rM(r, L, \be)$ where $\be = ((r+1)\epsilon, r\epsilon)$ for a small $0 < \epsilon \ll 1$. Moreover, unless $\ell = 1$ and $j = -1$, it lies on the interior of the effective cone. (We will treat $\ell = 1$, $j = -1$ case in \textbf{Step 4.}) Thus, if we take $\ba$ as (possibly a slight perturbation of) the one associated to $\cO(r+1, r) \otimes \Theta^{j+2}$, that is, $(\frac{r+1}{r(j+2)+\ell}, \frac{r}{r(j+2)+\ell})$ by \eqref{eqn:weighttobundle}, $\cO(r+1, r) \otimes \Theta^{j+2}$ is nef and big on $\rM(r, L, \ba)$. By Kawamata-Viehweg vanishing, $\rH^{i}(\rM(r, L, \ba), \cO(r+1, r) \otimes \Theta^{j+2} \otimes \omega) = 0$ for $i > 0$. Thus, it is enough to show that $\rH^{i}(\rM(r, L, \ba), \cO(1, 0) \otimes \Theta^{j}) \cong \rH^{i}(\rM(r, L, \be), \cO(1, 0) \otimes \Theta^{j})$.

\textbf{Step 2.} We can move from $\be$ to $\ba$ by a nearly diagonal wall-crossing (Section \ref{ssec:diagonal}). All walls that we encounter are simple wall $\Delta(s, e, (s, 1))$ (Proposition \ref{prop:diagonalissimple}). The wall $\Delta(s, e, (s, 1))$ is given by $sa_{1} + (r-s)a_{2} = sd - re$. So if the wall actually occurs while we move from $\be$ to $\ba$, 
\begin{equation}\label{eqn:boundforj}
	sd - re < s\frac{r+1}{r(j+2)+\ell} + (r-s)\frac{r}{r(j+2)+\ell} = \frac{r^{2}+s}{r(j+2)+\ell}. 
\end{equation}
The first wall occurs when $sd - re = 1$. In this case, $s = \ell$. Thus, if $\frac{r^{2}+\ell}{r(j+2)+\ell} < 1$, or equivalently, if $j > r-2$, we do not cross any wall. Then $\rM(r, L, \ba) \cong \rM(r, L, \be)$ and we are done. Thus, it is sufficient to show $\rH^{i}(\rM(r, L, \ba), \cO(1, 0) \otimes \Theta^{j}) \cong \rH^{i}(\rM(r, L, \be), \cO(1, 0) \otimes \Theta^{j})$ for $-1 \le j \le r-2$. 

\textbf{Step 3.} For each wall $\Delta(s, e, (s, 1))$, let $\lambda_{-}$ be the associated one-parameter subgroup. Combining Proposition \ref{prop:thetalambdaweight} and Lemma \ref{lem:wallimplieszeroweight}, we have 
\[
	\wt_{\lambda_{-}}\left(\cO(1, 0) \otimes \Theta^{j}\right) = \chi\left(\frac{s}{r} - \left(\frac{\ell}{r} + j\right)(sd - re)\right).
\]
Since $sd - re < r$, for any wall, it is straightforward to check that $\wt_{\lambda_{-}}\left(\cO(1, 0) \otimes \Theta^{j}\right) \le \wt_{\lambda_{-}}\left(\cO(1, 0) \otimes \Theta^{-1}\right)  < \chi(s/r + (sd - re)) < \eta_{+}$ for any $j \ge -1$ by comparing with \eqref{eqn:etap}.

Now we need to show that 
\begin{equation}\label{eqn:etamineq}
	-\eta_{-} < \wt_{\lambda_{-}}\left(\cO(1, 0) \otimes \Theta^{j}\right)
\end{equation}
for every wall $\Delta(s, e, (s, 1))$ with 
\[
	sd - re < \frac{r^{2} + s}{r(j+2)+\ell}.
\]
Equation \eqref{eqn:etamineq} is equivalent to 
\[
	(\frac{\ell}{r} + j - 1)(sd - re) < s(r-s) + \frac{s}{r}
\]
and this is trivial for $j \le 0$. Let $1 \le j \le r-2$. Then 
\[
	(\frac{\ell}{r} + j - 1)(sd - re) < j(sd - re) < \frac{j(r^{2}+s)}{r(j+2) +\ell} < \frac{(j+2)(r^{2}+s)}{r(j+2)} = r + \frac{s}{r} \le s(r-s) + \frac{s}{r}
\]
provided $2 \le s \le r-2$. If $s = 1$, 
\[
	\frac{j(r^{2}+1)}{r(j+2)+\ell} < \frac{(r-2)(r^{2}+1)}{r(r-2+2)} < r -1 + \frac{1}{r}.
\]
Finally, if $s = r-1$, 
\[
	\frac{j(r^{2}+r-1)}{r(j+2)+\ell} < \frac{(r-2)(r^{2}+r-1)}{r(r-2+2)} < r - 1 + \frac{r-1}{r}.
\]
In any case, we have the inequality \eqref{eqn:etamineq}. Therefore, by Theorem \ref{thm:cohomologyidentification}, 
\[
	\rH^{i}(\rM(r, L, \ba), \cO(1, 0) \otimes \Theta^{j}) 
	\cong \rH^{i}(\rM(r, L, \be), \cO(1, 0) \otimes \Theta^{j}). 
\]

\textbf{Step 4.} The only remaining case is that $\ell = 1$ (hence $d = 1$) and $j = -1$. We need to prove $\rH^{i}(\rM(r, L), \cE_{x} \otimes \Theta^{-1}) = 0$ for $i > 0$. Since $d = 1$, there is a contraction map $\pi_{1} : \PP(\cE_{x}) = \rM(r, L, r-1, \epsilon) \to \rM(r, L(-x))$ (Remark \ref{rmk:nef}). Then by \cite[Lemma 13]{BM19}, 
\[
\begin{split}
	\rH^{i}(\rM(r, L), \cE_{x} \otimes \Theta^{-1}) &\cong \rH^{i}(\PP(\cE_{x}), \cO(1) \otimes \Theta^{-1}) = \rH^{i}(\PP(\cE_{x}), \omega_{\PP(\cE_{x})} \otimes \cO(r+1))\\
	&= \rH^{i}(\PP(\cE_{x}), \omega_{\PP(\cE_{x})} \otimes \pi_{1}^{*}\Theta_{\rM(r, L(-x))}^{r+1}).
\end{split} 
\]

By Koll\'ar's vanishing (\cite[Theorem 2.1]{Kol86}), $R^{i}\pi_{1 *}\omega_{\PP(\cE_{x})}$ is torsion free for all $i$ and 
\[
	\rH^{k}(\rM(r, L(-x)), R^{i}\pi_{1 *}\omega_{\PP(\cE_{x})} \otimes \Theta_{\rM(r, L(-x))}^{r+1}) = 0
\]
for all $k > 0$. Since the Leray spectral sequence degenerates, $\rH^{0}(\rM(r, L(-x)), R^{i}\pi_{1 *}\omega_{\PP(\cE_{x})} \otimes \Theta_{\rM(r, L(-x))}^{r+1}) \cong \rH^{i}(\PP(\cE_{x}), \omega_{\PP(\cE_{x})} \otimes \pi_{1}^{*}\Theta_{\rM(r, L(-x))}^{r+1})$. On the other hand, over the stable locus $\rM(r, L(-x))^{s}$, $\pi_{1}$ is a $\PP^{r-1}$-fibration. Checking a general fiber, we can show that $R^{i}\pi_{1 *}\omega_{\PP(\cE_{x})} = 0$ for $i \ne r - 1$. Thus, we obtain the desired vanishing for $1 \le i \le r-2$. 

For $i = r-1$, since $R^{r-1} \pi_{1 *}\omega_{\PP(\cE_{x})}$ is a torsion free sheaf, we have an injective morphism $R^{r-1} \pi_{1 *}\omega_{\PP(\cE_{x})} \hookrightarrow (R^{r-1} \pi_{1 *}\omega_{\PP(\cE_{x})})^{\vee\vee}$. These two are isomorphic to $\omega_{\rM(r, L(-x))}$ over an open subset of codimension $\ge 2$ (\cite[Exercise III.8.4]{Har77}) and the latter is reflexive. Since $\rM(r, L(-x))$ is locally factorial (\cite[Theorem A]{DN89}), $(R^{r-1} \pi_{1 *}\omega_{\PP(\cE_{x})})^{\vee\vee} \cong \omega_{\rM(r, L(-x))} \cong \Theta_{\rM(r, L(-x))}^{-2r}$ (\cite[Theorem F]{DN89}). We have 
\[
\begin{split}
	\rH^{0}(\rM(r, L(-x)), R^{r-1}\pi_{1 *}\omega_{\PP(\cE_{x})} \otimes \Theta_{\rM(r, L(-x))}^{r+1}) &\hookrightarrow \rH^{0}(\rM(r, L(-x)), \omega_{\rM(r, L(-x))} \otimes \Theta_{\rM(r, L(-x))}^{r+1})\\
	&= \rH^{0}(\rM(r, L(-x)), \Theta_{\rM(r, L(-x))}^{-r+1}) = 0.
\end{split}
\]
\end{proof}

\begin{remark}\label{rmk:g=2}
When $g = r = 2$, $\rM(r, L)$ is an intersection of two quadrics in $\PP^5$ and $\cE_{x}$ is a spinor bundle \cite{CKL19, FK18}. From this description, it was shown that $\cE_{x}$ is ACM for all $x \in X$.
\end{remark}


\section{Semiorthogonal decomposition}\label{sec:SOD}

Since $\rM(r, L)$ is an index two Fano variety of Picard number one (\cite{Ram73}), $\cO, \Theta$ form an exceptional collection. In this section, we prove Theorem \ref{thm:SOD} by showing that the exceptional collection and the image of $\rD^{b}(X)$ form a part of a semiorthogonal decomposition of $\rD^{b}(\rM(r, L))$. It was proved for $r = 2$ in \cite{Nar17, Nar18}, and for $d = 1$ and $g \ge 3r + 4$ in \cite{BM19}. Since a stronger version of Theorem \ref{thm:SOD} is proved for $r = 2$ (\cite[Theorem 1.1]{TT21}), we assume that $r \ge 3$. 

\begin{proof}[\protect{Proof of Theorem \ref{thm:SOD}}]
By Theorem \ref{thm:mainthmembedding}, we have four full subcategories
\[
	\cO, \Phi_{\cE}(\rD^{b}(X)), \Theta, \Phi_{\cE}(\rD^{b}(X))\otimes \Theta.
\]
We will show that they are semiorthogonal in that order. We need to prove the orthogonality condition. Since $\{ \CC(x) ~ | ~ x \in X \}$ form a spanning class of $\rD^{b}(X)$, $\{ \cE_{x} ~ | ~ x \in X \}$ (resp. $\{\cE_{x}\otimes \Theta ~ | ~ x \in X \}$) form a spanning class of $\Phi_{\cE}(\rD^{b}(X))$ (resp. $\Phi_{\cE}(\rD^{b}(X))\otimes \Theta$). Therefore, it is sufficient to prove the cohomology vanishing in Theorem \ref{thm:orthogonality} below.
\end{proof}

\begin{theorem}\label{thm:orthogonality}
Assume the $g(X) \ge 6$. For any $i \in \ZZ$ and not necessarily distinct two points $x_{1}, x_{2} \in X$, the following cohomologies are trivial.
\begin{enumerate}
\item $\rH^i(\rM(r,L), \cE^*_{x_{1}})$;
\item $\rH^i(\rM(r,L), \cE_{x_{1}} \otimes \Theta^{-1})$;
\item $\rH^i(\rM(r,L), \cE_{x_{1}} \otimes \cE^*_{x_{2}} \otimes \Theta^{-1})$. 
\end{enumerate}
\end{theorem}

We use Sommese's vanishing theorem for $k$-ample vector bundles. On a smooth variety $V$, a line bundle $A$ on $V$ is {\em $k$-ample }if it is semiample and the dimension of the fiber of the morphism $|mA| : V \to \PP^N$ is less than or equal to $k$ for $m \gg 0$. A vector bundle $F$ is $k$-ample if $\cO_{\PP(F)}(1)$ is $k$-ample.

\begin{theorem}[\protect{\cite[Proposition 1.13]{Som78}}, Sommese vanishing theorem]\label{thm:Sommese}
Let $F$ be a rank $r$ $k$-ample vector bundle on $V$. Then we have $\rH^i(V,\omega_V \otimes F)=0$ for $i \ge r+k$.
\end{theorem}

\begin{lemma}\label{lem:kamplenessElargeell}
Suppose that $\ell \ge 2$ and the first wall crossing is a simple one. Over $\rM(r, L)$, $\cE_{x}$ is $(g-1)\ell(r-\ell)$-ample.
\end{lemma}

\begin{proof}
Recall that $\PP(\cE_{x}) \cong \rM(r, L, r-1, \epsilon)$. The first wall-crossing arises when the parabolic weight is $1/\ell$ and the wall is of the form $\Delta(\ell, e, \ell)$ (Lemma \ref{lem:1stwallcrossing}). The associated strictly nef line bundle is $\cO_{\PP (\cE_{x})}(1)$ (Lemma \ref{lem:intersection2}). Thus, it is sufficient to compute the dimension of the exceptional fiber of $\pi_{-} : \rM(r, L, r-1, \epsilon) \to \rM(r, L, r-1, 1/\ell)$. 

For a point $p := ((E^{+}, V^{+}) \oplus (E^{-}, V^{-})) \in \rM(r, L, r-1, 1/\ell)$, $\pi_{-}^{-1}(p) = \PP \Ext^{1}((E^{-}, V^{-}), (E^{+}, V^{+}))$. We can compute its dimension, by modifying the exact sequence \eqref{eqn:les}. After a standard computation, we obtain $\dim \PP \Ext^{1}((E^{-},V^{-}), (E^{+}, V^{+})) = (g-1)\ell(r-\ell)$.
\end{proof}

\begin{lemma}\label{lem:kamplenessEtensorE}
For any two points $x_{1}, x_{2} \in X$, $\cE_{x_{1}} \otimes \cE_{x_{2}}^{*}\otimes \Theta$ is $(g-1)\ell(r-\ell)$-ample. 
\end{lemma}

\begin{proof}
First, suppose that $2 \le \ell \le r-2$. Note that $\cE_{x_{2}}^{*}\otimes \Theta$ is a normalized Poincar\'e bundle over $\rM(r, L^{*}(r))$, where $\deg L^{*}(r) = r - d$. For the wall crossing of $\PP(\cE_{x_{1}}) \cong \rM(r, L, r-1, \epsilon)$, the first wall is $\Delta(\ell, e, \ell)$ and it is a multiple wall if and only if $2\ell < r$ (Lemma \ref{lem:1stwallcrossing}). On the other hand, $\PP (\cE_{x_{2}}^{*}\otimes \Theta) \cong \rM(r, L^{*}(r), r-1, \epsilon)$ and its first wall is $\Delta(r-\ell, e', r-\ell)$ and it is a multiple wall if and only if $2(r-\ell) < r$. But since $2\ell + 2(r-\ell) = 2r$ and $\ell \nmid r$, one of these two walls is simple. Then we may apply Lemma \ref{lem:kamplenessElargeell} to compute the $k$-ampleness of one of them. By \cite[Corollary 3.5]{LN20}, we can conclude that $\cE_{x_{1}} \otimes \cE_{x_{2}}^{*}\otimes \Theta$ is (at least) $(g-1)\ell(r-\ell)$-ample.

Now suppose $\ell = 1$ ($\ell = r-1$ case is the same). By \cite[Proposition 21]{BM19}, $\cE_{x_{1}}$ is $\frac{r(r-1)}{2}g$-ample. On the other hand, for $\PP(\cE_{x_{2}}^{*}\otimes \Theta) \cong \rM(r, L^{*}(r), r-1, \epsilon)$, the first wall is $\Delta(r-1, e, r-1)$. Since $2(r-1) > r$ (because $r > 2$), this is a simple wall, so $\cE_{x_{2}} \otimes \Theta$ is $(g-1)\ell (r-\ell) = (g-1)(r-1)$-ample by Lemma \ref{lem:kamplenessElargeell}. Since $(g-1)(r-1) < \frac{r(r-1)}{2}g$, $\cE_{x_{1}} \otimes \cE_{x_{2}}^{*}\otimes \Theta$ is $(g-1)\ell(r-\ell)$-ample, too.
\end{proof}

\begin{proof}[Proof of Theorem \ref{thm:orthogonality}]
We first show Item (2). For $i \ne 0$, it follows from Theorem \ref{thm:vanishingoftwistedbundle}. From 
\[
	\rH^{0}(\rM(r, L), \cE_{x_{1}} \otimes \Theta^{-1}) \cong \rH^{0}(\rM(r, L, \be), \cO(1, 0) \otimes \Theta^{-1})
\]
and the fact that $\cO(1, 0) \otimes \Theta^{-1} = \Theta^{\frac{\ell-1}{r} - 1} \otimes \left(\cO(r, 0) \otimes \Theta^{1-\ell}\right)^{\frac{1}{r}}$ is not on $\Eff(\rM(r, L, \be))$ (because $\ell < r$), it is trivial. Thus, we obtain Item (2). Since $\cE_{x_{1}}^{*} = \cE_{x_{1}}^{*}\otimes \Theta \otimes \Theta^{-1}$ and $\cE_{x_{1}}^{*}\otimes \Theta$ is the normalized Poincar\'e bundle on $\rM(r, L^{*}(r)) \cong \rM(r, L)$, Item (1) follows from Item (2). 

We move to Item (3). By Theorem \ref{thm:Sommese} and Lemma \ref{lem:kamplenessEtensorE}, 
\[
	\rH^{i}(\rM(r, L), \cE_{x_{1}} \otimes \cE_{x_{2}}^{*}\otimes \Theta^{-1}) \cong \rH^{i}(\rM(r, L), \cE_{x_{1}} \otimes \cE_{x_{2}}^{*}\otimes \Theta \otimes \omega) = 0
\]
if $i \ge r^{2} + (g-1)\ell(r-\ell)$. Serre duality tells us that 
\[
	\rH^{i}(\rM(r, L), \cE_{x_{1}} \otimes \cE_{x_{2}}^{*}\otimes \Theta^{-1}) \cong \rH^{(r^{2}-1)(g-1)-i}(\rM(r, L), \cE_{x_{1}}^{*}\otimes \cE_{x_{2}} \otimes \Theta^{-1})^{*} = 0
\]
provided $i \le (r^{2}-1-\ell(r-\ell))(g-1) - r^{2}$. Thus, if $r^{2} + (g-1)\ell(r-\ell) \le (r^{2}-1-\ell(r-\ell))(g-1)-r^{2}+1$, we obtain the desired vanishing. This is equivalent to 
\begin{equation}\label{eqn:genusbound}
	\frac{2r^{2}-1}{r^{2}-1-2\ell(r-\ell)} \le g-1.
\end{equation}
Since $0 < \ell < r$, if $r \ge 5$, we have 
\[
	\frac{2r^{2}-1}{r^{2}-1-2\ell(r-\ell)} < \frac{2r^{2}-1}{r^{2} - 1 - r^{2}/2} = \frac{4r^{2}-1}{r^{2}-2} \le \frac{4\cdot 5^{2}-1}{5^{2}-2} = \frac{99}{23}.
\]
When $r = 3$ and $4$, a direct computation gives $17/4$ and $31/9$, respectively. So if $g \ge 6$, the inequality \eqref{eqn:genusbound} holds for all $r \ge 3$.
\end{proof}

In the proof above, the genus bound is necessary only for Item (3), which is used to prove the orthogonality of $\Phi_{\cE}(\rD^{b}(X)) \otimes \Theta$. Thus, we obtain the following weaker version for all $g \ge 2$. 

\begin{theorem}\label{thm:threefactors}
There is a semiorthogonal decomposition $\rD^{b}(\rM(r, L)) = \langle \cA', {}^{\perp}\cA' \rangle$ where $\cA' = \langle \cO, \Phi_{\cE}(\rD^{b}(X)), \Theta\rangle$. 
\end{theorem}

\begin{remark}\label{rmk:smallgenus}
Whenever $r$ and $\ell$ satisfy \eqref{eqn:genusbound}, we have the semiorthogonal decomposition of Theorem \ref{thm:SOD}. Therefore the genus bound can be improved, if we restrict $\deg L$. For instance, if $d = 1$ (so $\ell = 1$) and $r \ge 6$, the vanishing result holds for $g \ge 4$. Our method does not seem to work for $g = 2, 3$. 
\end{remark}

\begin{remark}\label{rmk:genusboud}
For $d = 1$, the semiorthogonal decomposition in Theorem \ref{thm:SOD} was obtained in \cite[Theorem B]{BM19}. Their genus bound is weaker than ours -- for instance, for $r \ge 4$, they proved it for $g \ge 3r + 4$. 
\end{remark}


\section{ACM bundles on $\rM(r, L)$}\label{sec:ACM}

Besides the structure of $\rD^{b}(\rM(r, L))$, another immediate application of the technique we developed in this paper is a construction of a one-dimensional family of ACM bundles. 

\begin{definition}\label{def:ACM}
Let $V$ be an $n$-dimensional projective variety with an ample line bundle $A$. A vector bundle $F$ on $V$ is an \emph{ACM bundle} with respect to $A$ if $\rH^i(V, F \otimes A^{j})=0$ for every $1 \leq i \leq n-1$ and $j \in \ZZ$. An ACM bundle $F$ is \emph{Ulrich} if $\rH^0(V, F \otimes A^{-1})=0$ and $\rH^0(V, F)=\rk F \cdot \deg V = \rk F \cdot (A)^{n}$.
\end{definition}

For a smooth Fano variety of Picard rank one, it is straightforward to verify that every line bundle is ACM. It is also clear that if $F$ is ACM with respect to $A$, then $F \otimes A^{k}$ is ACM with respect to $A$ for all $k \in \ZZ$. But finding a non-trivial example of an ACM bundle is not an easy task for higher dimensional varieties. In this section, we show that $\cE_{x}$ is ACM.

\begin{remark}\label{rmk:veryample}
Many authors assume that $A$ to be very ample when they consider ACM bundles. Because the Picard number of $\rM(r, L)$ is one, Theorem \ref{thm:ACM} implies that $\cE_x$ is ACM for every very ample line bundle. On $\rM(r, L)$, $\Theta^{k}$ is known to be very ample when $k \ge r^{2}+r$ (\cite[Theorem A]{EP04}), but an optimal $k$ for the very ampleness is unknown.
\end{remark}

\begin{proof}[Proof of Theorem \ref{thm:ACM}]
By Serre duality, 
\[
	\rH^{i}(\rM(r, L), \cE_{x} \otimes \Theta^{j}) \cong \rH^{\dim \rM(r, L)-i}(\rM(r, L), \cE_{x}^{*} \otimes \Theta \otimes \Theta^{-j-3})^{*}.
\]
The vanishing for $\cE_{x} \otimes \Theta^{j}$ for $j \le -2$ follows from the vanishing for $\cE_{x}^{*} \otimes \Theta \otimes \Theta^{j}$ for $j \ge -1$. Since $\cE_{x}^{*} \otimes \Theta$ is the normalized Poincar\'e bundle over $\rM(r, L^{*}(r)) \cong \rM(r, L)$, it is sufficient to prove the vanishing for $j \ge -1$, which is Theorem \ref{thm:vanishingoftwistedbundle}.

For two different points $x_{1}, x_{2} \in X$, $\cE_{x_{1}} \ne \cE_{x_{2}}$ (\cite[Theorem]{LN05}). Thus, we obtain a one-dimensional family of ACM bundles. 
\end{proof}

\begin{remark}\label{rmk:Ulrich}
The bundle $\cE_{x}$ is not Ulrich in general. If $g = r = 2$, $h^0(\rM(r, L),\cE_x)= 4 < 8 = 2 \deg(\rM(r, L))$. It is an interesting problem to construct Ulrich bundles on $\rM(r,L)$. See \cite{CKL19} for an alternative construction of Ulrich bundles for $g=r=2$ case.
\end{remark}


\bibliographystyle{alpha}

\end{document}